\documentclass{amsproc}

\usepackage{amssymb}

\usepackage{graphicx}

\usepackage{caption,mathrsfs,mathtools,psfrag,relsize,tikz,upgreek,xcolor,xparse}

\newcommand*\circled[1]{\tikz[baseline=(char.base)]{
   \node[shape=circle,draw,inner sep=1pt] (char) {#1};}}
\newcommand\cf[1]{\circled{\makebox[1.2cm]{#1}}}

\NewDocumentCommand\DownArrow{O{2.0ex} O{black}}{%
   \mathrel{\tikz[baseline] \draw [<-, line width=0.5pt, #2] (0,0) -- ++(0,#1);}
}

\makeatletter
\let\c@equation\c@figure
\makeatother

\numberwithin{equation}{section}

\newtheorem{thm}[equation]{Theorem}

\theoremstyle{definition}
\newtheorem{dfn}[equation]{Definition}
\newtheorem{ex}[equation]{Example}

\theoremstyle{remark}
\newtheorem{rem}[equation]{Remark}

\newcommand{\thmref}[1]{Theorem~\ref{#1}}

\newcommand{\dfnref}[1]{Definition~\ref{#1}}
\newcommand{\remref}[1]{Remark~\ref{#1}}
\newcommand{\exref}[1]{Example~\ref{#1}}
\newcommand{\figref}[1]{Figure~\ref{#1}}
\newcommand{\secref}[1]{Section~\ref{#1}}

\makeatletter
\newcommand{\myboxed}[1]{\text{\fboxsep=.2em\fbox{\m@th$\displaystyle#1$}}}
\makeatother

\newcommand{\A}{{\mathcal{A}}}
\newcommand{\ab}{{\boldsymbol{a}}}
\newcommand{\al}{{\alpha}}
\newcommand{\au}{\upalpha}

\newcommand{\B}{{\mathcal{B}}}
\newcommand{\bb}{{\boldsymbol{b}}}
\newcommand{\be}{\beta}
\newcommand{\bs}{\backslash}

\newcommand{\C}{\operatorname{\mathfrak Cay}}
\newcommand{\Cb}{{\mathbf C}}
\newcommand{\cb}{{\boldsymbol{c}}}
\newcommand{\CC}{\mathcal C}
\newcommand{\cc}{\circlearrowright}

\newcommand{\D}{\Delta}
\newcommand{\de}{\delta}
\newcommand{\du}{\mathop{\uparrow\!\downarrow}}

\newcommand{\E}{\mathcal{E}}

\newcommand{\F}{\Phi}

\newcommand{\fun}{\operatorname{fun}}

\newcommand{\ii}{\!\!}

\newcommand{\jj}{\!\!\!\!}

\newcommand{\G}{\Gamma}
\newcommand{\g}{\gamma}

\newcommand\ka{\kappa}

\newcommand{\Lc}{\mathcal{L}}

\newcommand{\mapstoto}{\mathop{\,\sim\joinrel\rightsquigarrow\,}}

\newcommand{\N}{\mathbb {N}}

\newcommand{\ora}{\overrightarrow}

\newcommand{\Pc}{\mathcal{P}}

\newcommand{\Sb}{{\mathbf S}}
\newcommand{\sd}{\rightthreetimes}
\newcommand{\Sch}{\operatorname{\mathfrak Sch}}
\newcommand{\Sf}{{\mathfrak S}}
\newcommand{\Si}{{\Sigma}}

\newcommand{\Ss}{{\mathscr{S}}}

\newcommand{\T}{{\mathbf T}}
\newcommand{\Tb}{{\mathbb T}}

\newcommand{\uu}{{\boldsymbol{u}}}

\newcommand{\V}{{\mathcal{V}}}
\newcommand{\vv}{{\boldsymbol{v}}}
\newcommand{\vn}{\varnothing}

\newcommand{\ww}{{\boldsymbol{w}}}
\newcommand{\wt}{\widetilde}

\newcommand{\Z}{{\mathbb Z}}
\newcommand{\z}{\zeta}

\begin{document}

\author{Vadim A.\ Kaimanovich}

\address{Department of Mathematics and Statistics, University of
Ottawa, 585 King Edward, Ottawa ON, K1N 6N5, Canada}

\email{vkaimano@uottawa.ca, vadim.kaimanovich@gmail.com}

\title[Circular slider graphs]{Circular slider graphs: de Bruijn, Kautz, Rauzy, lamplighters and spiders}

\subjclass[2010]{Primary 05C20, 20E22; Secondary 05C80, 37B10, 68R15}

\keywords{De Bruijn graphs, lamplighter groups}

\begin{abstract}
We suggest a new point of view on de Bruijn graphs and their subgraphs based on using circular words rather than linear ones.
\end{abstract}

\maketitle

\thispagestyle{empty}

\section*{Introduction}

\subsection{De Bruijn graphs}

De Bruijn graphs represent overlaps between consecutive subwords of the same length in a longer word. Under various names and in various guises they and their subgraphs currently enjoy a lot of popularity in mathematics (dynamical systems and combinatorics) as well as in the applications to computer science (data networks) and bioinformatics (DNA sequencing).

A very succinct description of these graphs can be found in the following two-line rhyme from the title of the 1975 de Bruijn's historical note \cite{deBruijn75}. The
$$
\parbox{.6\linewidth}{\emph{Circular arrangements of $2^n$ zeros and ones\\ That show each $n$-letter word exactly once}}
$$
he is talking about are precisely the Hamiltonian cycles (currently known as \textsf{de Bruijn sequences}) \footnotemark \; in the \textsf{de Bruijn graph
$$
\ora\B_{\ii\A}^n=\ora\B_{\ii|\A|}^n
$$
of span $n$}. This is the directed graph (\textsf{digraph}) whose vertices are all $n$-letter words $\ab=\al_1 \al_2 \dots \al_n$ in a given finite alphabet $\A$ (quite often $\A$ is just the binary alphabet $\{0,1\}$, like in de Bruijn's formulation above), and whose arrows (directed edges)
%$$
\begin{equation} \label{eq:usual}
\ab=\al_1 \al_2 \dots \al_n \mapstoto \ab'=\al_2 \dots \al_n\al_{n+1}
\end{equation}
%$$
are the pairs of $n$-words with a length $n-1$ overlap, so that the associated transitions (we call them \textsf{de Bruijn transitions}) consist in removing the initial letter $\al_1$ of an $n$-word and adding instead a new letter $\al_{n+1}$ at the end of the word.

\footnotetext{\;Throughout the paper we use \textsf{sans serif} when giving a definition or introducing a notation, whereas \emph{italic} is used for emphasizing and in quotes (as usual) or when mentioning a certain term for the first time without defining it.}

\subsection{Circular vs linear}

The purpose of this note is to suggest a new point of view on de Bruijn graphs and their subgraphs based on using \emph{circular words} rather than \emph{linear} ones. Although the idea of circularity (of de Bruijn sequences) has been present in the subject area ever since the very first known formulation of this setup in 1894 (see \secref{sec:dB} for more historical details), and the notions of \emph{necklaces} and \emph{Lyndon words} play a pivotal role in various algorithms for generating de Bruijn sequences (e.g., see Fredricksen -- Maiorana \cite{Fredricksen-Maiorana78}, Perrin -- Restivo \cite{Perrin-Restivo15}, or the latest Sawada -- Williams -- Wong \cite{Sawada-Williams-Wong16, Sawada-Williams-Wong17}, and the references therein), it has never been applied to the vertex $n$-letter words themselves. They have always been treated as linear words with the ``giving'' and ``receiving'' ends (the one that loses a letter and the one that acquires a new letter in the process of a de Bruijn transition, respectively) being $n$ symbols apart.~\footnote{\;The only exception we are aware of is a recent article by B\"ohmov\'a -- Dalf\'o -- Huemer \cite{Bohmova-Dalfo-Huemer15} and the ensuing papers by Dalf\`o \cite{Dalfo17a,Dalfo17} where \emph{cyclic Kautz graphs} were introduced and studied.}

Our approach is based on the totally obvious observation that there is a one-to-one correspondence between linear words and \textsf{pointed circular words} of the same length (the \textsf{pointer} separates the initial and the final letter of the linear word written clockwise). In a more formal language we replace the $\Z$-valued indices $i=1,2,\dots,n$ which parameterize the letters $\al_i$ of a word $\ab=\al_1 \al_2 \dots \al_n$ with the $\Z_n$-valued indices $\iota = i \;(\!\!\!\!\mod n)$, so that the pointer is positioned between the letters with the indices $\iota=n\equiv 0 \;(\!\!\!\!\mod n)$ and $\iota=1$ $(\!\!\!\!\mod n)$. In this interpretation de Bruijn transitions \eqref{eq:usual} consist in moving the pointer one position clockwise and (possibly) changing the letter located between the old and the new positions of the pointer. Equivalently, moving the pointer along a fixed word is the same as moving the whole word in the opposite direction with respect to a fixed pointer, so that in terms of the \textsf{circular shift} (anticlockwise rotation)
%$$
\begin{equation} \label{eq:shift}
(\Sb\ab)_\iota = (\ab)_{\iota+1 \;(\!\!\!\!\!\!\mod n)} \;,\qquad \ab\in\A^n\cong\A^{\Z_n} \;,
\end{equation}
%$$
de Bruijn transitions are
%$$
\begin{equation} \label{eq:cyclic}
\ab \mapstoto \ab' \iff (\ab')_\iota = (\Sb\ab)_\iota \quad\forall\, \iota\neq 0
\iff
\al'_\iota=\al_{\iota+1 \;(\!\!\!\!\!\!\mod n)} \quad\forall\, \iota\neq 0
\;.
\end{equation}
%$$

\subsection{Slider graphs}

Actually, we find it more convenient to think, instead of a pointer, about a sliding window \textsf{(slider)} of width 2 which covers one letter on either side of the pointer (we call a circular word endowed with a window like this \textsf{slider pointed}). From this point of view the usual de Bruin graph $\ora\B_{\ii\A}^n=\ora\B_{\ii|\A|}^n$ can be identified with the \textsf{full circular slider graph
of span~$n$ over an alphabet $\A$}. This is the digraph whose vertices are all slider pointed circular words of length $n$ in alphabet~$\A$, and whose arrows are the \textsf{slider transitions} which consist in moving the slider one position clockwise and (possibly) replacing the letter at the intersection of the old and the new slider windows, see \figref{fig:intro}.

\begin{figure}[h]
\begin{center}
\psfrag{a}[cl][cl]{{\rotatebox[origin=c]{60}{$\!\!\!\!\!\al_{n-1}$}}}
\psfrag{b}[cl][cl]{{\rotatebox[origin=c]{30}{$\al_n$}}}
\psfrag{c}[cl][cl]{$\al_1$}
\psfrag{d}[cl][cl]{{\rotatebox[origin=c]{-30}{$\al_2$}}}
\psfrag{e}[cl][cl]{{\rotatebox[origin=c]{-55}{$\al_3$}}}
\psfrag{v}[cl][cl]{{\rotatebox[origin=c]{85}{$\!\!\!\!\cdots$}}}
\psfrag{w}[cl][cl]{{\rotatebox[origin=c]{-85}{$\cdots$}}}
\hfil
\includegraphics[scale=.15]{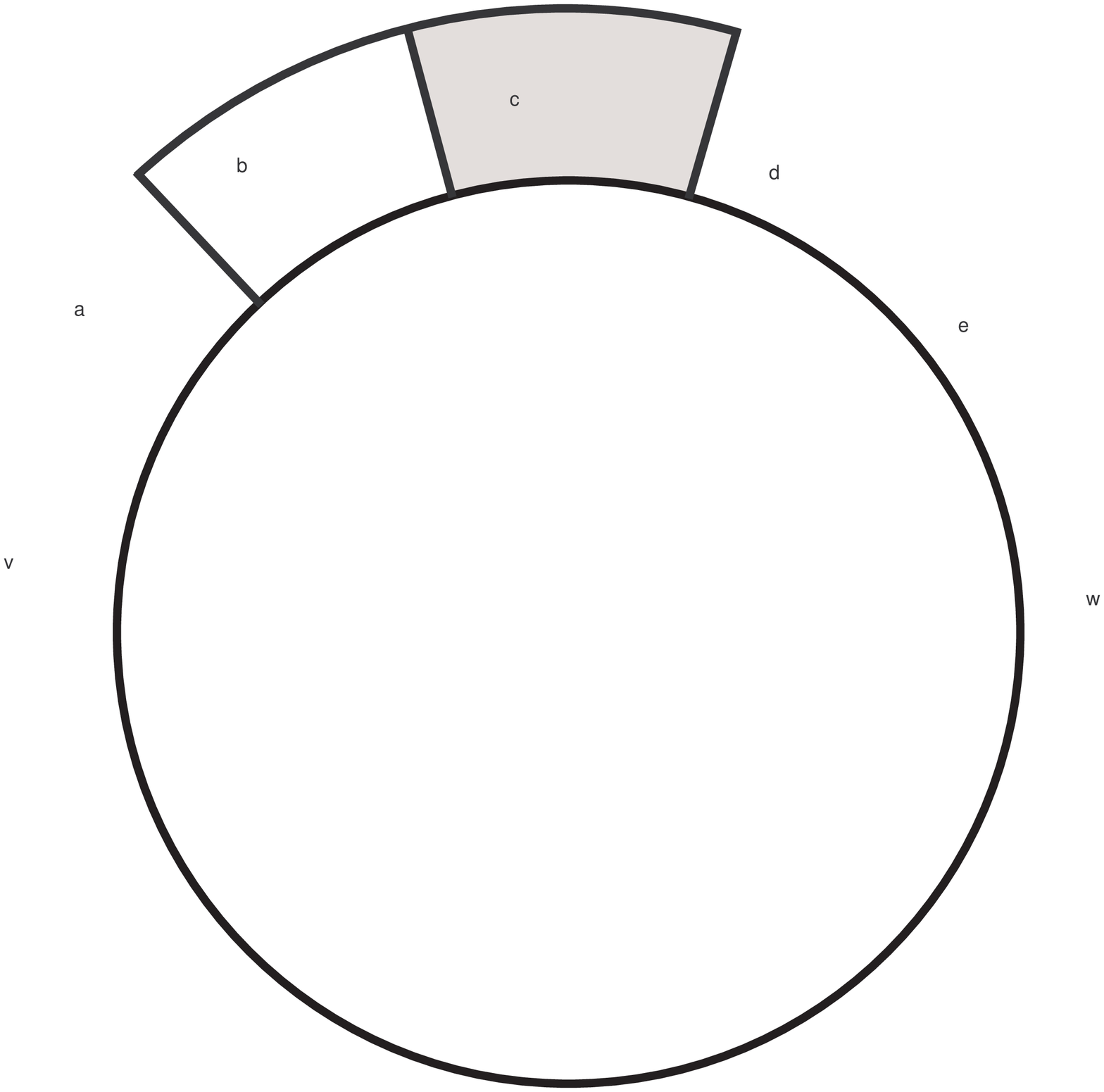}
\hfil
\raisebox{15mm}{\text{\Large$\mapstoto$}}
\hfil
\psfrag{c}[cl][cl]{$\!\!\!\al_{n+1}$}
\includegraphics[scale=.15]{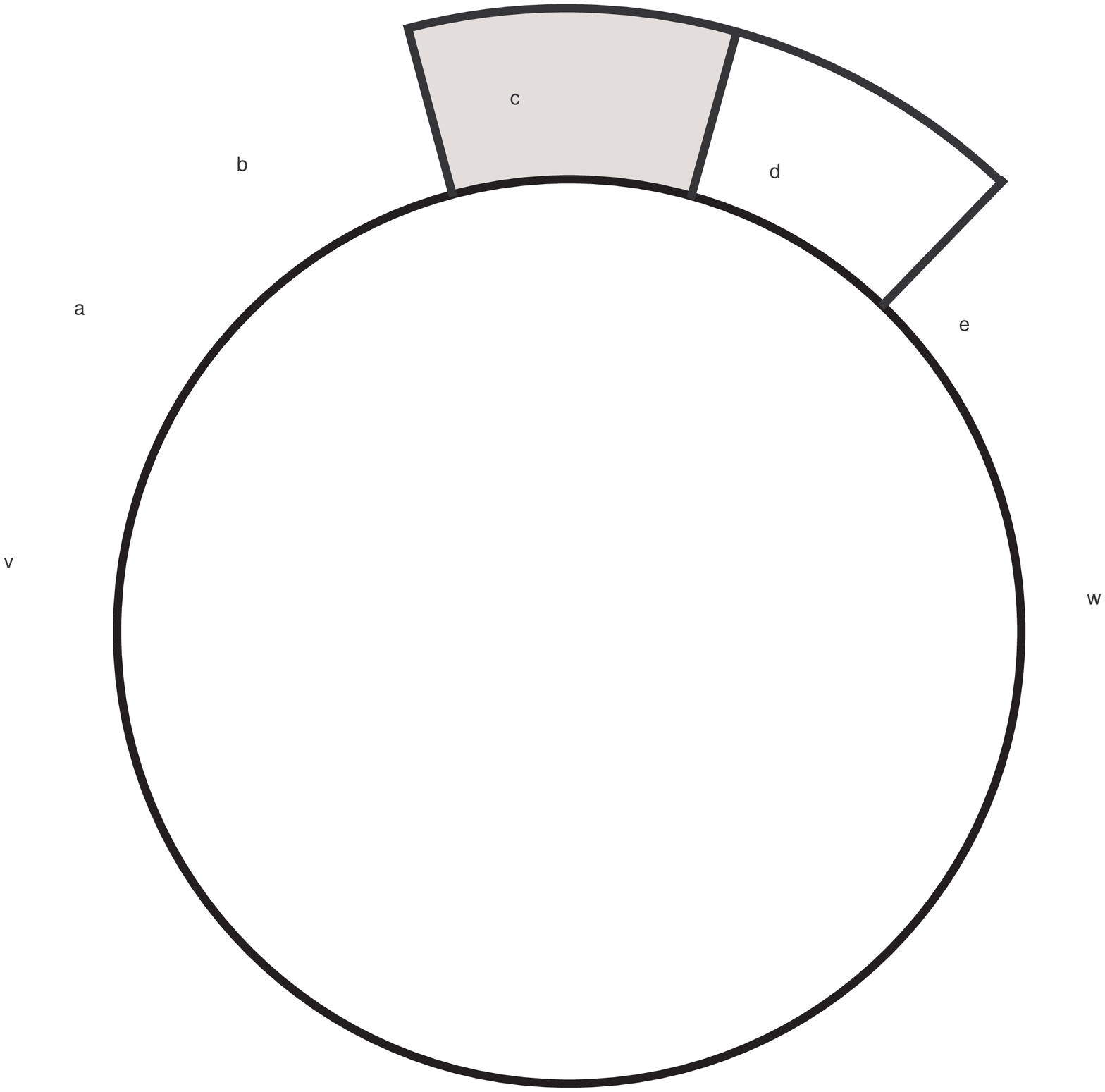}
\hfil\hfil\hskip 1cm
\end{center}
\vskip 1cm
\begin{center}
\psfrag{a}[cl][cl]{{\rotatebox[origin=c]{60}{$\!\!\!\al_{-1}$}}}
\psfrag{b}[cl][cl]{{\rotatebox[origin=c]{30}{$\al_0$}}}
\psfrag{c}[cl][cl]{$\al_1$}
\psfrag{d}[cl][cl]{{\rotatebox[origin=c]{-30}{$\al_2$}}}
\psfrag{e}[cl][cl]{{\rotatebox[origin=c]{-55}{$\al_3$}}}
\psfrag{v}[cl][cl]{{\rotatebox[origin=c]{85}{$\!\!\!\!\cdots$}}}
\psfrag{w}[cl][cl]{{\rotatebox[origin=c]{-85}{$\cdots$}}}
\hfil
\includegraphics[scale=.15]{sliderleft1.eps}
\hfil
\raisebox{15mm}{\text{\Large$\mapstoto$}}
\hfil
\psfrag{c}[cl][cl]{$\al'_1$}
\includegraphics[scale=.15]{sliderright1.eps}
\hfil\hfil\hskip 1cm
\end{center}
\caption{Slider transitions between circular $n$-words in the ``conventional'' parameterization with letter indices $i\in\Z$ and in the cyclic parameterization with letter indices $\iota\in\Z_n$.}
\label{fig:intro}
\end{figure}

A general (directed) \textsf{circular slider graph} is then defined as a subgraph of the full circular slider graph. In particular, this class includes the \textsf{induced circular slider graphs} $\ora\Ss[\V]$ obtained by restricting the full circular slider graph $\ora\B_{\ii\A}^n$ to a subset of vertices $\V\subset\A^n$. One can also further modify circular slider graphs by introducing various \emph{decorations} or \emph{labellings} of their vertices and arrows.

\subsection{Motivation}

The motivation for our work comes from two sources. The first one is the isomorphism of the Cayley graphs $\C(\Z\wr\Z_m,Q)$ of the \emph{lamplighter groups}
$$
\Z\wr\Z_m = \Z \sd \fun(\Z,\Z_m)
$$
(here $\fun(\Z,\Z_m)$ denotes the additive group of finitely supported $\Z_m$-valued configurations on $\Z$) with the corresponding \textsf{Diestel -- Leader graphs} ($\equiv$ \emph{horospheric products} $\Tb_{m+1}\du\Tb_{m+1}$ of two copies of pointed at infinity \emph{homogeneous trees} $\Tb_{m+1}$ of degree $m+1$) discovered by M\"oller, Neumann and Woess (see the historical accounts by Woess \cite{Woess13} and Amchislavska -- Riley \cite{Amchislavska-Riley15}). The generating set~$Q$ of the group $\Z\wr\Z_m$, for which this isomorphism is established, consists of the set $Q_+$ of the \textsf{walk-right---switch generators}
$$
\left(1,\de_1^b\right) = (1,\vn) \cdot \left(0,\de_0^b\right)
$$
and of the set $Q_-$ of their inverses (the \textsf{switch---walk-left generators})
$$
\left(1,\de_1^b\right)^{-1}  = \left(0,\de_0^{-b}\right) \cdot (-1,\vn)= \left(-1,\de_0^{-b}\right)
$$
(where $b$ runs through $\Z_m$, the zero configuration on $\Z$ is denoted by $\vn$, and $\de_i^b$ denotes the configuration on $\Z$ with $\de_i^b(i)=b$ and $\de_i^b(j)=0$ for $j\neq i$). In order to make the symmetry in the definition of these generators more explicit, we shall ``mark'' the integer line $\Z$ with, instead of lamplighter's position $z\in\Z$, the width 2 \textsf{slider} (``lamplighter's window'') over the positions $z$ and $z+1$. Then the multiplication by an element $\left(1,\de_1^b\right)\in Q_+$ (resp., by its inverse $\left(-1,\de_0^{-b}\right)\in Q_-$) amounts to moving the slider one position to the right (resp., to the left) and adding $b$ (resp.,~$-b$) to the state at the intersection of the old and the new sliders. These are precisely the slider transitions (and their inverses) like on \figref{fig:intro} above between configurations ($\equiv$ strings) from $\fun(\Z,\Z_m)$.

The second source is the link between de Bruijn graphs and lamplighter groups recently discovered and studied by Grigorchuk -- Leemann -- Nagnibeda \cite{Grigorchuk-Leemann-Nagnibeda16} and Leemann \cite{Leemann16} who, for a fixed alphabet size $m$, identified the \emph{Benjamini -- Schramm limits} of de Bruijn graphs $\ora\B_{\ii m}^n$ as $n\to\infty$ with the corresponding directed versions of the Diestel -- Leader graphs $\Tb_{m+1}\du\Tb_{m+1}$. Their approach is algebraic, and it is based on a presentation of the de Bruijn graph $\ora\B_{\ii m}^n$ as the \emph{Schreier graph} of a natural spherically transitive action of the associated lamplighter group $\Z\wr\Z_m$ on the $n$-th level of the $m$-regular rooted tree. Moreover, by using the classification of the subgroups of $\Z\wr\Z_m$ obtained by Grigorchuk -- Kravchenko \cite{Grigorchuk-Kravchenko14}, they also identify the tensor ($\equiv$ direct) products of de Bruijn graphs~$\ora\B_{\ii m}^n$ and \textsf{directed cycle graphs} $\ora\CC_{\ii k}$ (these products are known as \textsf{spider-web graphs}) with the Cayley graphs of appropriate finite groups for a number of combinations of the parameters $k,m,n$. In particular, for $k=n$ the spider-web graph $\ora\CC_{\ii n}\otimes\ora\B_{\ii m}^n$ turns out to be isomorphic to the Cayley digraph $\ora\C(\Z_n\wr\Z_m,Q_+)$ of the finite lamplighter group $\Z_n\wr \Z_m$ with respect to the set $Q_+$ of the walk-right---switch generators.

The immediate origin of the present paper is the project in progress \cite{Kaimanovich-Leemann-Nagnibeda} (joint with Paul-Henry Leemann and Tatiana Nagnibeda) aimed at understanding to what extent the results on Benjamini -- Schramm limits from \cite{Grigorchuk-Leemann-Nagnibeda16, Leemann16} could be carried over to generalizations of de Bruijn graphs, and in a sense the present paper can be considered as prolegomena to \cite{Kaimanovich-Leemann-Nagnibeda}. The work on this project started during my gratefully acknowledged visit to the University of Geneva in February-March 2017.

\subsection{New perspective} \label{sec:perspective}

Of course, formally the circular slider graphs are just subgraphs of the corresponding full de Bruijn graphs, and a number of examples of this kind were considered before (we review them in \secref{sec:circularexamples}). However, our point of view brings in a new perspective.

\medskip

\noindent
$\bullet$\;
We make the aforementioned relationship between de Bruijn graphs and finite lamplighter groups completely transparent. Indeed, as we have already explained, the arrows in $\ora\C(\Z\wr\Z_m,Q_+)$ are precisely our slider transitions between infinite words. The same observation holds for the lamplighter groups $\Z_n\wr\Z_m$ over a finite cyclic group $\Z_n$ as well (with, \emph{mutatis mutandis}, $Q$ and $Q_\pm$ now denoting the corresponding subsets of $\Z_n\wr\Z_m$), so that the only difference between the de Bruijn ($\equiv$ slider) graph $\ora\B_{\ii m}^n$ and the Cayley digraph $\ora\C(\Z_n\wr\Z_m,Q_+)$ is that in the case of slider graphs the length $n$ circular words in the alphabet $\Z_m$ are marked only once (with the position of the slider), whereas in the case of lamplighter groups the circular words are endowed with two pointers (both the position of the identity of the group $\Z_n$ and the position of the lamplighter). Therefore, in order to describe the slider graphs in terms of the lamplighter group one has to eliminate the additional pointer, or, in other words, to pass from the Cayley digraph $\ora\C(\Z_n\wr\Z_m,Q_+)$ to the corresponding \emph{Schreier digraph}
$\ora\Sch\left(\Z_n\bs\left(\Z_n\wr\Z_m\right),Q_+\right)$ on the quotient of the wreath product $\Z_n\wr\Z_m$ by the cyclic subgroup of translations $\{(z,\vn):z\in\Z_n\}\cong\Z_n$ acting on the left. This Schreier digraph is nothing else than the de Bruijn graph~$\ora\B_{\ii m}^n$, and, conversely, the tensor product $\ora\CC_{\ii n}\otimes \ora\B_{\ii m}^n$ is precisely the Cayley digraph $\ora\C(\Z_n\wr\Z_m,Q_+)$.

\medskip

\noindent
$\bullet$\; Our approach gives rise to new interesting classes of circular slider graphs (i.e., of subgraphs of full de Bruijn graphs) inherently linked with the presence of a circular structure. In \secref{sec:periodic} we introduce the \textsf{periodic slider graphs}, which are the induced subgraphs on subsets $\V\subset\A^n$ invariant with respect to the circular shift $\Sb$ \eqref{eq:shift}. In \secref{sec:transMark} we introduce the \textsf{transversally Markov circular slider graphs}, for which one imposes an additional admissibility condition on the replacements $\al_1\mapstoto\al_{n+1}$ in the slider transitions from \figref{fig:intro}, i.e., $\al_1\mapstoto\al_{n+1}=\al_1'$ has to be an arrow of a certain digraph ($\equiv$ topological Markov chain) on the alphabet~$\A$. In the particular case when $\A$ is endowed with the structure of a Cayley or a Schreier digraph this definition gives rise to
what we call \textsf{Cayley} and \textsf{Schreier circular slider graphs}, respectively.

\medskip

\noindent
$\bullet$\;
In spite of the enormous popularity of de Bruijn graphs and their various modifications, there have been very few attempts to extend this notion to infinite graphs. Laarhoven -- de Weger \cite{Laarhoven-deWeger13} in the course of a discussion of a link between de Bruijn graphs and the famous $3n+1$ Collatz conjecture (see \exref{ex:collatz}) introduced the infinite \emph{$p$-adic de Bruijn graph} with the vertex set $\A^{\Z_+}$ and the arrows $\ww\mapstoto \Sb\ww$, where $p=|\A|$, and $\Sb:\al_0\al_1\dots\mapsto\al_1\al_2\dots$ is the (unilateral) shift transformation on $\A^{\Z_+}$.
However, this graph does not really grasp a number of significant features of finite de Bruijn graphs, as, for instance, the out-degree of any vertex is 1.

On the other hand, as we have noticed when talking about the slider interpretation of the Cayley graph of the lamplighter group $\Z\wr\Z_m$, an advantage of our approach is that the definition of slider transitions is applicable without much difference both to finite circular and to infinite linear words. Therefore, by letting the span $n$ in the definition of circular slider graphs go to infinity, one naturally arrives at the notion of the \textsf{full linear slider graph} $\ora\Ss_{\jj\A}$ as the digraph whose vertex set is $\A^\Z$ and whose arrows are the slider transitions from the bottom half of \figref{fig:intro} with ``$n=\infty$''. General \textsf{linear slider graphs} are further defined as the subgraphs of~$\ora\Ss_{\jj\A}$. We shall return to a discussion of linear slider graphs and their relationship with the horospheric products of trees, lamplighter groups over $\Z$ and their Schreier graphs elsewhere.

\medskip

\noindent
$\bullet$\;
Neighbourhoods in circular and in linear slider graphs look precisely the same, which provides a direct approach to the aforementioned result of Grigorchuk -- Leemann -- Nagnibeda \cite{Grigorchuk-Leemann-Nagnibeda16, Leemann16} on the identification of the Benjamini -- Schramm limits of de Bruijn graphs with Diestel -- Leader graphs. This idea can actually be used for the identification of the Benjamini -- Schramm limits for much more general sequences of circular slider graphs \cite{Kaimanovich-Leemann-Nagnibeda}.

\medskip

\noindent
$\bullet$\;
Alternatively, one can directly consider the \textit{stochastic homogenization} for linear slider graphs by looking for equivalence relations with an invariant probability measure graphed by linear slider graphs (cf. \cite{Kaimanovich03a}). Let us remind that two strings $\ab,\bb\in\A^\Z$ are called \textsf{(asynchronously) asymptotically equivalent} if there exist $\D_-,\D_+\in\Z$ such that the strings $\Sb^{\D_-}\ab$ (resp., $\Sb^{\D_+}\ab$) and $\bb$ are cofinal at $-\infty$ (resp., at $+\infty$).
If, additionally, $\D_-=\D_+$ (resp., $\D_-=\D_+=0$), then the strings $\ab$ and $\bb$ are called \textsf{semi-synchronously} (resp., \textsf{synchronously}) \textsf{asymptotically equivalent}. The synchronous asymptotic equivalence relation is also known as the \textsf{homoclinic} or \textsf{Gibbs equivalence relation} of the shift transformation~$\Sb$, and it preserves the maximal entropy invariant measure for any subshift of finite type on the alphabet $\A$ \cite{Petersen-Schmidt97}. Since the semi-synchronous asymptotic equivalence relation is the common refinement of the synchronous one and of the orbit equivalence relation of the shift $\Sb$, it also preserves these measures. The connected components of linear slider graphs are clearly contained in the equivalence classes of the semi-synchronous equivalence relation, which provides a natural stochastic homogenization for the linear slider graphs induced on subshifts of finite type in $\A^\Z$.

\subsection{Paper overview}

We begin with a brief discussion of the rich history of de Bruijn graphs (\secref{sec:dB}). In \secref{sec:circular} we introduce the notion of a circular slider graph, after which in \secref{sec:circularexamples} we review the known examples of subgraphs of de Bruijn graphs from this point of view. In \secref{sec:periodic} we define periodic slider graphs and establish their $d$-connectedness for arbitrary irreducible sofic subshifts (\thmref{thm:connect}). In \secref{sec:transMark} we introduce and discuss transversally Markov circular slider graphs. In \secref{sec:lamplighter} we remind the basic definitions concerning the lamplighter groups, after which in \secref{sec:lampslider} we prove that de Bruijn digraphs and, more generally, Cayley circular slider digraphs can be interpreted as Schreier digraphs of circular lamplighter groups for a suitable choice of generating subsets (\thmref{thm:sch} and \thmref{thm:schC}). Finally, in \secref{sec:spiderslider} we introduce spider slider graphs and identify a certain class of them with the appropriate Cayley graphs of circular lamplighter groups (\thmref{thm:spider}).

\subsection*{Acknowledgements} I would like to thank Reinhard Diestel, Rostislav Grigorchuk, Paul-Henry Leemann and Tatiana Nagnibeda for helpful and inspiring discussions. I am also grateful to Paul-Henry Leemann for valuable comments on a preliminary version of this paper. Last but not least, my thanks go to Florian Sobieczky, the editor of this volume, for all his effort and patience.

\section{De Bruijn graphs} \label{sec:dB}

Let
$$
\A^n = \{ \al_1 \al_2 \dots \al_n : \al_i\in \A \}
$$
denote the \textsf{set of words}
$$
\ab=\al_1 \al_2 \dots \al_n
$$
of \textsf{length} $n$ in a \textsf{finite alphabet} $\A$.

\begin{dfn}
Two words $\ab,\ab'\in \A^n$ are linked with a \textsf{de Bruijn transition}
%$$
\begin{equation} \label{eq:dBtrans}
\ab\mapstoto \ab'
\end{equation}
%$$
if there exists a word $\wt\ab\in\A^{n+1}$ such that $\ab$ and $\ab'$ are its initial and final segments, respectively:
$$
\wt\ab = \lefteqn{\overbracket{\phantom{\al_1 \al_2 \dots \al_{d-1} \al_n\,}}^\ab} \al_1\underbracket{\al_2 \dots \al_{n-1} \al_n \, \al_{n+1}}_{\ab'} \;.
$$
\end{dfn}

\begin{dfn}[e.g., see \cite{Allouche-Shallit03,Rigo14}] \label{dfn:rauzy}
The \textsf{span $n$ de Bruijn digraph}
$$
\ora\B_{\ii\A}^n=\ora\B_{\ii m}^n
$$
over an alphabet~$\A$ of cardinality $|\A|=m$ is the directed graph \emph{(digraph)} whose vertex set is $\A^n$, and whose arrows are de Bruijn transitions \eqref{eq:dBtrans}. Sometimes one forgets the orientation on edges and considers \textsf{undirected de Bruijn graphs} $\B_{\ii\A}^n$ as well.
\end{dfn}

These graphs are called after Nicolaas Govert de Bruijn who in 1946 constructed \cite{deBruijn46} what is now known as a \emph{de Bruijn binary sequence} (the formal definition of which is contained in the title of \cite{deBruijn75}), i.e., a \emph{Hamiltonian path} in $\ora\B_{2}^n$. Actually, as it was later discovered by Richard Stanley, in the binary case this had already been done as early as in 1894 (the problem was formulated by de Rivi\`ere and solved by Flye Sainte-Marie), and the case of an arbitrary alphabet had also been treated before de Bruijn in 1934 by Martin (the general case was also independently considered in 1946 by Good with a follow-up by Rees). Still, the current generally accepted term for these graphs is \emph{de Bruijn graphs}. It appears to be quite fair as it was de Bruijn himself who in 1975 wrote a very detailed historical note \cite{deBruijn75}, in which all these works and several other related ones were discussed. This note was eloquently and explicitly called
\begin{quotation}
\textit{Acknowledgement of priority to {C}. {F}lye {S}ainte-{M}arie on the counting of circular arrangements of $2^n$ zeros and ones that show each $n$-letter word exactly once},
\end{quotation}
an \textit{avis rarissima} by today's standards. By the way (the fact which is not mentioned in \cite{deBruijn75}), de Bruijn binary sequence was also apparently independently introduced in the famous \emph{The Logic of Scientific Discovery} by Karl Popper \cite[Appendix iv]{Popper59} under the name of a \emph{shortest random-like sequence}.

De Bruijn graphs are currently quite popular in computer science and bioinformatics, e.g., see \cite{BangJensen-Gutin09}, \cite{Gross-Yellen-Zhang14}.

\section{Circular slider graphs} \label{sec:circular}

We shall now somewhat change the viewpoint and pass from the usual ``linear'' words to the ones written
around a circle (and read clockwise). The position of the initial letter of a word will be marked with a pointer separating the final and the initial letters. The circle itself is allowed to rotate freely, so that the whole word is entirely determined just by the mutual positions of the letters and of the pointer, and therefore there is a natural one-to-one correspondence between \textsf{linear words} and \textsf{pointed circular words} of the same length, see \figref{fig:markerpointed}.

\begin{figure}[h]
\begin{center}
\hfil
\raisebox{15mm}{$\al_1 \al_2 \dots \al_n$}
\hfil
\raisebox{15mm}{\text{\Large$\rightleftarrows$}}
\psfrag{a}[cl][cl]{{\rotatebox[origin=c]{60}{$\!\!\!\!\!\al_{n-1}$}}}
\psfrag{b}[cl][cl]{{\rotatebox[origin=c]{30}{$\al_n$}}}
\psfrag{c}[cl][cl]{$\al_1$}
\psfrag{d}[cl][cl]{{\rotatebox[origin=c]{-30}{$\al_2$}}}
\psfrag{e}[cl][cl]{{\rotatebox[origin=c]{-55}{$\al_3$}}}
\psfrag{v}[cl][cl]{{\rotatebox[origin=c]{85}{$\!\!\!\!\cdots$}}}
\psfrag{w}[cl][cl]{{\rotatebox[origin=c]{-85}{$\cdots$}}}
\hfil
\includegraphics[scale=.15]{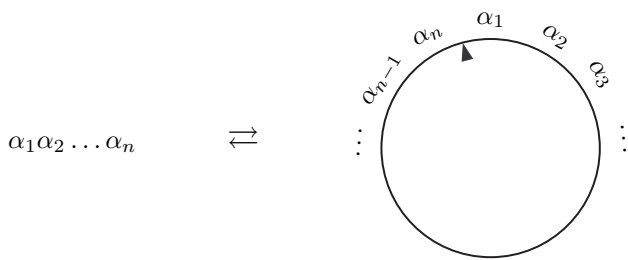}
\hfil\hfil\hskip 1cm
\end{center}
\caption{Linear words and pointed circular words.}
\label{fig:markerpointed}
\end{figure}

This observation is begging to ``change the coordinate system'' and to pass in the definition of de Bruijn transitions \eqref{eq:dBtrans} from \emph{moving the letters with respect to a fixed pointer} in the ``pointer coordinate system'', as on \figref{fig:lettersmove}, to \emph{moving the pointer in the opposite direction with respect to fixed letters} in the ``letters coordinate system'' instead, as on \figref{fig:pointermoves}.

\begin{figure}[h]
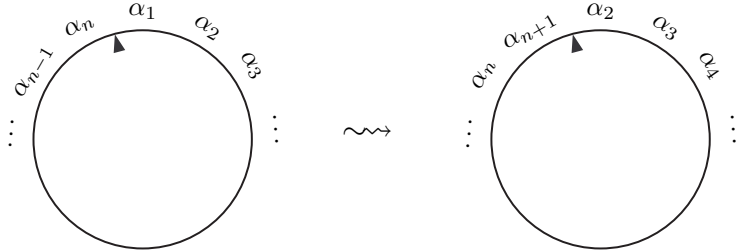

\begin{center}
\psfrag{a}[cl][cl]{{\rotatebox[origin=c]{60}{$\!\!\!\!\!\al_{n-1}$}}}
\psfrag{b}[cl][cl]{{\rotatebox[origin=c]{30}{$\al_n$}}}
\psfrag{c}[cl][cl]{$\al_1$}
\psfrag{d}[cl][cl]{{\rotatebox[origin=c]{-30}{$\al_2$}}}
\psfrag{e}[cl][cl]{{\rotatebox[origin=c]{-55}{$\al_3$}}}
\psfrag{v}[cl][cl]{{\rotatebox[origin=c]{85}{$\!\!\!\!\cdots$}}}
\psfrag{w}[cl][cl]{{\rotatebox[origin=c]{-85}{$\cdots$}}}
\hfil
\includegraphics[scale=.15]{pointerleft.eps}
\hfil
\raisebox{15mm}{\text{\Large$\mapstoto$}}
\hfil
\psfrag{a}[cl][cl]{{\rotatebox[origin=c]{65}{$\!\!\!\!\!\!\al_n$}}}
\psfrag{b}[cl][cl]{{\rotatebox[origin=c]{30}{$\!\!\!\!\!\al_{n+1}$}}}
\psfrag{c}[cl][cl]{$\al_2$}
\psfrag{d}[cl][cl]{{\rotatebox[origin=c]{-30}{$\al_3$}}}
\psfrag{e}[cl][cl]{{\rotatebox[origin=c]{-55}{$\al_4$}}}
\includegraphics[scale=.15]{pointerleft.eps}
\hfil\hfil\hskip 1cm
\end{center}
\caption{De Bruijn transitions with respect to a fixed pointer.}
\label{fig:lettersmove}
\end{figure}

\begin{figure}[h]
\begin{center}
\psfrag{a}[cl][cl]{{\rotatebox[origin=c]{60}{$\!\!\!\!\!\al_{n-1}$}}}
\psfrag{b}[cl][cl]{{\rotatebox[origin=c]{30}{$\al_n$}}}
\psfrag{c}[cl][cl]{$\al_1$}
\psfrag{d}[cl][cl]{{\rotatebox[origin=c]{-30}{$\al_2$}}}
\psfrag{e}[cl][cl]{{\rotatebox[origin=c]{-55}{$\al_3$}}}
\psfrag{v}[cl][cl]{{\rotatebox[origin=c]{85}{$\!\!\!\!\cdots$}}}
\psfrag{w}[cl][cl]{{\rotatebox[origin=c]{-85}{$\cdots$}}}
\hfil
\includegraphics[scale=.15]{pointerleft.eps}
\hfil
\raisebox{15mm}{\text{\Large$\mapstoto$}}
\hfil
\psfrag{c}[cl][cl]{$\!\!\!\al_{n+1}$}
\includegraphics[scale=.15]{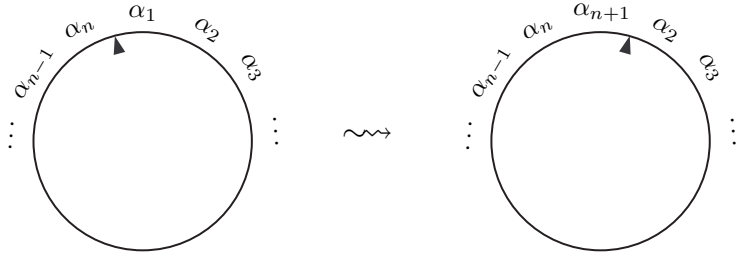}
\hfil\hfil\hskip 1cm
\end{center}
\caption{De Bruijn transitions with respect to a moving pointer.}
\label{fig:pointermoves}
\end{figure}

Then in terms of pointed circular words de Bruijn transitions consist in moving the pointer one position clockwise and (possibly) changing the letter located between the old and the new positions of the pointer. Actually, we find it more convenient to think, instead of a pointer, about a sliding window \textsf{(slider)} of width~2 which covers one letter on either side of the pointer, see \figref{fig:sliderpointed}.

\begin{figure}[h]
\begin{center}
\raisebox{15mm}{$\al_1 \al_2 \dots \al_n$}
\hfil
\raisebox{15mm}{\text{\Large$\rightleftarrows$}}
\psfrag{a}[cl][cl]{{\rotatebox[origin=c]{60}{$\!\!\!\!\!\al_{n-1}$}}}
\psfrag{b}[cl][cl]{{\rotatebox[origin=c]{30}{$\al_n$}}}
\psfrag{c}[cl][cl]{$\al_1$}
\psfrag{d}[cl][cl]{{\rotatebox[origin=c]{-30}{$\al_2$}}}
\psfrag{e}[cl][cl]{{\rotatebox[origin=c]{-55}{$\al_3$}}}
\psfrag{v}[cl][cl]{{\rotatebox[origin=c]{85}{$\!\!\!\!\cdots$}}}
\psfrag{w}[cl][cl]{{\rotatebox[origin=c]{-85}{$\cdots$}}}
\hfil
\includegraphics[scale=.15]{pointerleft.eps}
\hfil
\raisebox{15mm}{\text{\Large$\rightleftarrows$}}
\psfrag{a}[cl][cl]{{\rotatebox[origin=c]{60}{$\!\!\!\!\!\al_{n-1}$}}}
\psfrag{b}[cl][cl]{{\rotatebox[origin=c]{30}{$\al_n$}}}
\psfrag{c}[cl][cl]{$\al_1$}
\psfrag{d}[cl][cl]{{\rotatebox[origin=c]{-30}{$\al_2$}}}
\psfrag{e}[cl][cl]{{\rotatebox[origin=c]{-55}{$\al_3$}}}
\psfrag{v}[cl][cl]{{\rotatebox[origin=c]{85}{$\!\!\!\!\cdots$}}}
\psfrag{w}[cl][cl]{{\rotatebox[origin=c]{-85}{$\cdots$}}}
\hfil
\includegraphics[scale=.15]{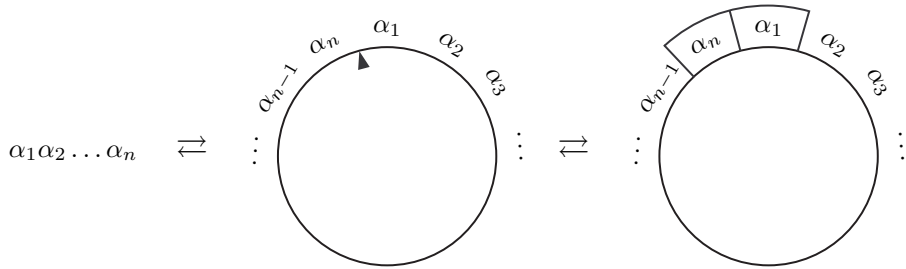}
\end{center}
\caption{Slider pointed circular words.}
\label{fig:sliderpointed}
\end{figure}

In this terminology de Bruijn transitions \eqref{eq:dBtrans} amount to moving the slider one position clockwise and (possibly) replacing the letter at the intersection of the old and the new slider windows (see \figref{fig:slidermoves}), and \dfnref{dfn:rauzy} of de Bruijn graphs $\ora\B_{\ii\A}^n=\ora\B_{\ii|\A|}^n$ can be then recast as

\begin{dfn} \label{dfn:slider}
(1) The \textsf{full circular slider graph of span~$n$ over an alphabet~$\A$} is the digraph whose vertices are all slider pointed circular words of length~$n$ in alphabet $\A$, and whose directed edges (arrows) are the de Bruijn transitions described on \figref{fig:slidermoves}.

(2) A (directed) \textsf{span $n$ circular slider graph} over an alphabet $\A$ is a subgraph of the full circular slider graph, i.e., its set of vertices is a subset of the set~$\A^n$ of the words of length~$n$, and its set of directed edges (arrows) is a subset of the set of de Bruijn transitions.

(3) For a subset $\V\subset\A^n$ we shall denote by $\ora\Ss[\V]$ the corresponding \textsf{induced circular slider graph}, i.e., the associated \emph{induced subgraph} of $\ora\B_{\ii\A}^n$ (its set of vertices is~$\V$, and its set of arrows consists of all de Bruijn transitions between these vertices).
\end{dfn}

\begin{figure}[h]
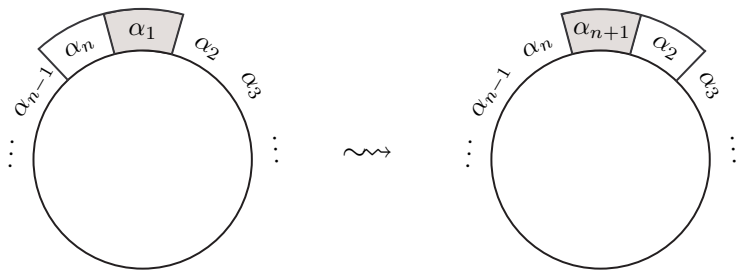

\begin{center}
\psfrag{a}[cl][cl]{{\rotatebox[origin=c]{60}{$\!\!\!\!\!\al_{n-1}$}}}
\psfrag{b}[cl][cl]{{\rotatebox[origin=c]{30}{$\al_n$}}}
\psfrag{c}[cl][cl]{$\al_1$}
\psfrag{d}[cl][cl]{{\rotatebox[origin=c]{-30}{$\al_2$}}}
\psfrag{e}[cl][cl]{{\rotatebox[origin=c]{-55}{$\al_3$}}}
\psfrag{v}[cl][cl]{{\rotatebox[origin=c]{85}{$\!\!\!\!\cdots$}}}
\psfrag{w}[cl][cl]{{\rotatebox[origin=c]{-85}{$\cdots$}}}
\hfil
\includegraphics[scale=.15]{sliderleft1.eps}
\hfil
\raisebox{15mm}{\text{\Large$\mapstoto$}}
\hfil
\psfrag{c}[cl][cl]{$\!\!\!\al_{n+1}$}
\includegraphics[scale=.15]{sliderright1.eps}
\hfil\hfil\hskip 1cm
\end{center}
\caption{De Bruijn transitions with respect to a moving slider.}
\label{fig:slidermoves}
\end{figure}

In the same way one can define \textsf{undirected circular slider graphs} as well. Note that \emph{a priori} we impose no \emph{connectedness conditions}. Although in the literature de Bruijn graphs and their subgraphs have so far always been considered over finite alphabets only, we do not impose any \emph{finiteness conditions} on the alphabet $\A$ either. In fact, there are meaningful examples of locally finite circular slider graphs over infinite alphabets as well (e.g., the \emph{Cayley} and \emph{Schreier circular slider graphs} introduced in \dfnref{dfn:slidergroup}).

\begin{rem}
Of course, formally the circular slider graphs from \dfnref{dfn:slider} are just subgraphs of the corresponding full de Bruijn graphs $\ora\B_{\ii\A}^n$, and numerous examples of this kind have been considered before (see \secref{sec:circularexamples} below). However, our point of view puts this notion in a new perspective as outlined in \secref{sec:perspective} of the Introduction.
\end{rem}

\begin{rem}
Our inspiration for choosing the term ``slider'' comes from the analogy with an old-fashioned mechanical analogue computer known as \textit{slide rule}, which was equipped with a sliding window split into two equal halves by a cursor line, see \figref{fig:rule} (this window was actually called ``runner'' though). It is interesting that there existed both linear and circular slide rules (cf. the comparison of the notions of linear and circular slider graphs in \secref{sec:perspective} of the Introduction)!
\end{rem}

\begin{figure}[h]
\begin{center}
\includegraphics[scale=.4]{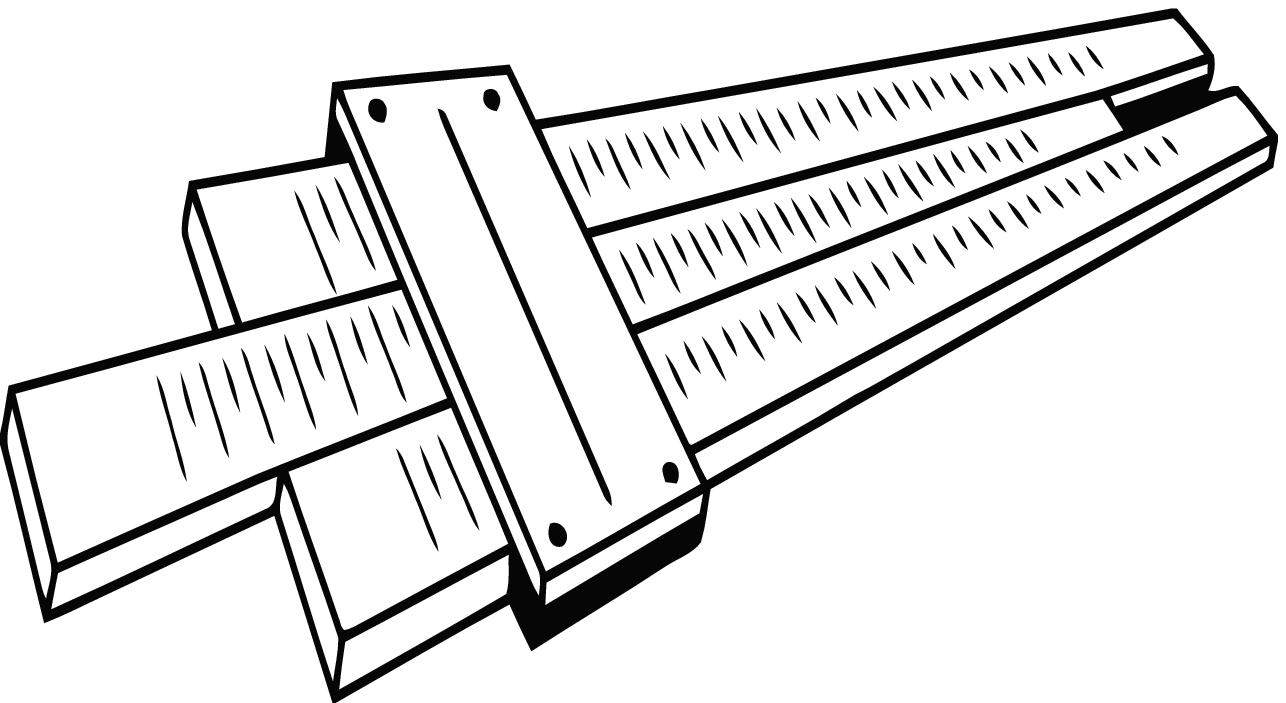} \hfil\hfil \includegraphics[scale=.15]{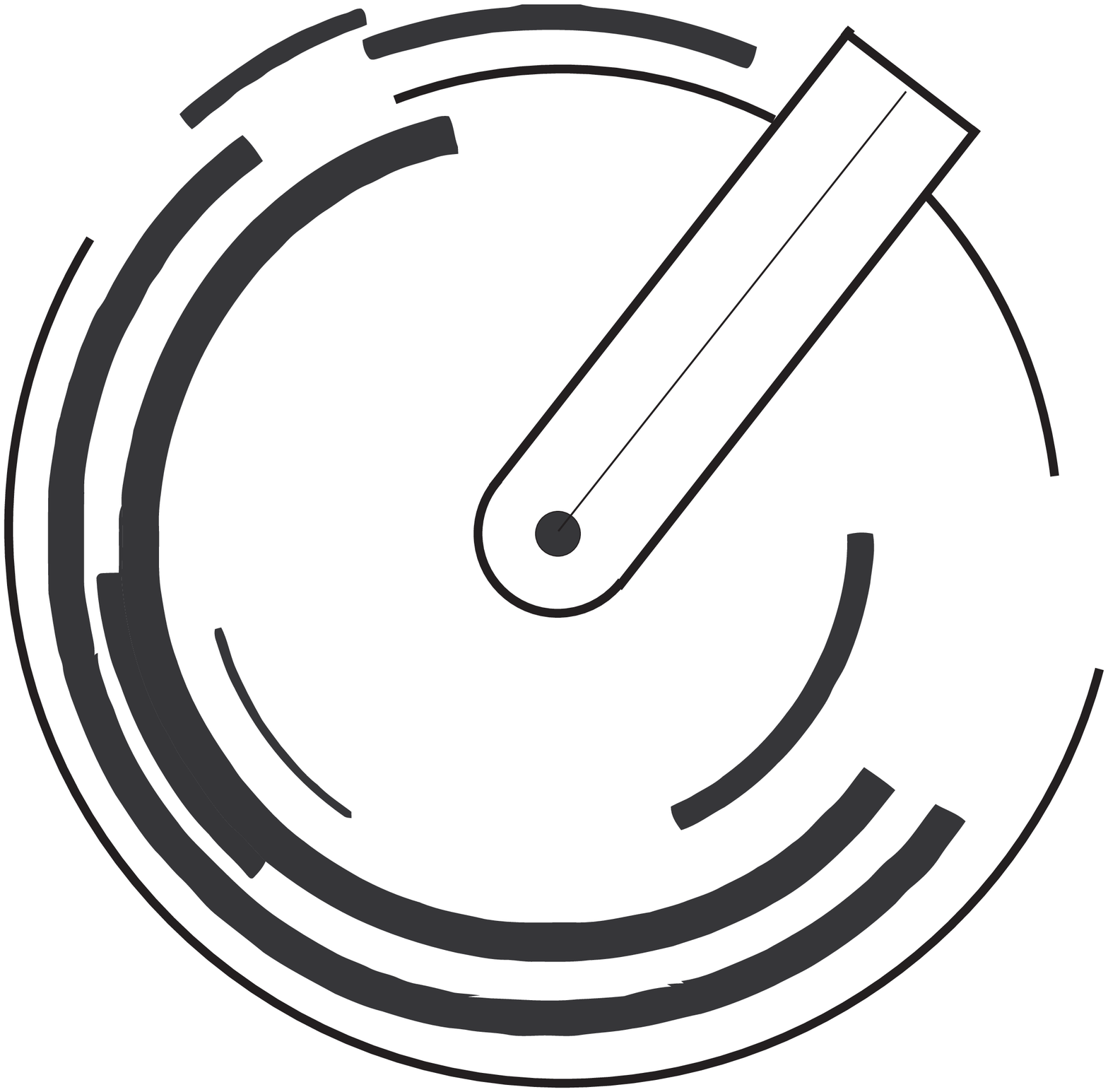} \hfil\hfil
\end{center}
\caption{Linear and circular slide rules.}
\label{fig:rule}
\end{figure}

\section{Examples} \label{sec:circularexamples}

General subgraphs of the de Bruijn graph $\ora\B_{\ii\A}^n$ (i.e., circular slider graphs in the sense of \dfnref{dfn:slider}) are parameterized by subsets $\E\subset\A^{n+1}$. The set of vertices of the associated subgraph is the collection of the length $n$ initial and final segments of the words from $\E$, and any $\ww\in\E$ produces an arrow between its length~$n$ initial and final segments \cite{Moreno2005a}.

Most examples of the subgraphs of de Bruijn graphs considered in the literature are determined by the set
$$
\Lc^{(n+1)}=\A^{n+1}\cap\Lc
$$
of the length $n+1$ words in a certain \textsf{factorial language} $\Lc$, i.e., such language that any \textsf{factor} ($\equiv$ truncation, both on the left and on the right) of a word from $\Lc$ is also in~$\Lc$. The set of vertices is then $\Lc^n$. We shall denote these graphs
$$
\ora\Ss^n(\Lc)
$$
and call them \textsf{factorial slider graphs}.

If the language $\Lc$ is in addition \textsf{prolongable} (i.e., for any $\ww\in\Lc$ there are $\al,\be\in\A$ such that $\al\ww\be\in\Lc$; the languages with this property are also called \textsf{extendable}), then $\Lc$ coincides with the language $\Lc(\Si)$ of the subwords of a certain \textsf{subshift} (closed shift invariant subset)
$$
\Si\subset\A^\Z
$$
(e.g., see \cite{Lind-Marcus95}), and, conversely, the language $\Lc(\Si)$ is obviously factorial and prolongable for any subshift $\Si$. In this situation
$$
\Lc^{(n)}=\Si^{(n)}
$$
is the set of the length $n$ subwords appearing in the infinite words from $\Si$, and we shall use the notation
$$
\ora\Ss^n(\Si) = \ora\Ss^n(\Lc)
$$
for the \textsf{slider graph of span $n$ determined by a subshift} $\Si$.

\begin{ex} \label{ex:DNA}
$\Lc=\Lc(\ww)$ is the language of the factors of a \emph{finite word} $\ww$ in the alphabet~$\A$. In bioinformatics one considers the situation when $\ww$ is a finite (if very long) \emph{DNA ``word''} (called \textsf{reference sequence}), and the associated graphs~$\ora\Ss^n(\Lc)=\ora\Ss^n(\ww)$ are called \textsf{de Bruijn graphs} (along with the usual full de Bruijn graphs $\ora\B_{\ii\A}^n$), see \cite{Compeau-Pevzner-Tesler2011}, \cite{Lin-Pevzner14}. One has also introduced the notion of a \textsf{colored de Bruijn graph} \cite{Iqbal12} to describe the situation when there are several reference sequences.
\end{ex}

\begin{ex} \label{ex:kautz}
$\Lc=\Lc(\Si)$ for the subshift $\Si\subset\A^\Z$ determined by the condition that \emph{any two consecutive letters in the words from $\Si$ are distinct} (equivalently, $\Lc$ is the language of the irreducible words representing the elements of the \emph{free product} of several copies of the cyclic 2-element group $\Z_2$ indexed by the alphabet $\A$). The associated graphs $\ora\Ss^n(\Si)$ are known as the \textsf{Kautz graphs} \cite{Kautz71}. Actually, in the case of a 3-letter alphabet $\A$ these graphs appear already in the original de Bruijn's paper \cite{deBruijn46}.
\end{ex}

\begin{ex} \label{ex:tmc}
$\Lc=\Lc(\Si)$ for the subshift $\Si$ corresponding to a \emph{topological Markov chain} \footnote{\;We distinguish general \textsf{subshifts of finite type} determined by a finite collection of \emph{forbidden words} and \textsf{topological Markov chains,} for which all forbidden words have length 2 (or, equivalently, all two-letter subwords must belong to a fixed set of \emph{admissible transitions}), e.g., see \cite{Lind-Marcus95}.} on the alphabet $\A$. The associated graphs $\ora\Ss^n(\Si)$ were introduced in 1983 by Fiol, Yebra and F\`abrega under the name of \textsf{sequence graphs} \cite{Fiol-Yebra-Fabrega83}, also see \cite{Gomez-Fiol-Yerba92} (however, these papers have remained virtually unknown outside of the Spanish graph theory community). If $\Si$ is the full shift, then, of course, $\ora\Ss^n(\Si)$ is the full de Bruin graph $\ora\B_{\ii\A}^n$.
\end{ex}

\begin{ex} \label{ex:sft}
The graphs $\ora\Ss^n(\Si)$ determined by (the language of) a \emph{general subshift of finite type} $\Si\subset\A^\Z$ were introduced by Moreno \cite{Moreno03} who was apparently not aware of
\cite{Fiol-Yebra-Fabrega83, Gomez-Fiol-Yerba92}.
\end{ex}

\begin{rem}
In \exref{ex:kautz} and \exref{ex:tmc} the graphs $\ora\Ss^n(\Lc)$ coincide with the induced subgraphs
$\ora\Ss\left[\Lc^{(n)}\right]$ determined by the sets of length $n$ words of the language $\Lc$. Provided $n$ is large enough, this is also the case for \exref{ex:tmc}.
\end{rem}

\begin{ex} \label{ex:rauzy}
$\Lc=\Lc(\ww)$ is the language of the factors of a \emph{single semi-infinite word} $\ww\in\A^{\Z_+}$. In this case the associated graphs $\ora\Ss^n(\Lc)$ were introduced by Rauzy \cite{Rauzy83} and are currently known as \textsf{Rauzy graphs}.
\end{ex}

\begin{ex} \label{ex:bosher}
$\Lc=\Lc(\Si)$ for a \emph{minimal subshift} $\Si\subset\A^\Z$ (then by minimality $\Lc=\Lc(\ww_+)$ for the positive subword $\ww_+$ of any $\ww\in\Si$, so that this setup is essentially equivalent to that of Rauzy from \exref{ex:rauzy}). The associated graphs~$\ora\Ss^n(\Si)$ were used by Boshernitzan in his analysis of the unique ergodicity for interval exchange maps \cite[Theorem 6.9]{Boshernitzan85}. Their appearance in \cite{Boshernitzan85} remained unknown outside of a very limited circle of specialists until 2010 when it was emphasized by Ferenczi -- Monteil in \cite{Ferenczi-Monteil10}.
\end{ex}

\begin{ex} \label{ex:fpl}
$\Lc$ is a \emph{general factorial prolongable language} in the alphabet~$\A$ (as we have already explained, this is the same as saying that $\Lc=\Lc(\Si)$ is the language of a \emph{subshift} $\Si\subset\A^\Z$). The associated graphs $\ora\Ss^n(\Si) = \ora\Ss^n(\Lc)$, considered as a generalization of the ones from \exref{ex:rauzy} above, were called \textsf{Rauzy graphs} by Cassaigne \cite{Cassaigne96} who used them for studying linear complexity languages.
\end{ex}

\begin{rem}
\exref{ex:kautz}, \exref{ex:tmc}, \exref{ex:sft}, \exref{ex:bosher} are particular cases of the situation when $\Si$ is a \emph{general subshift over an alphabet} $\A$ which is described in \exref{ex:fpl}. As is typical for this area (cf. de Bruijn's acknowledgement \cite{deBruijn75} and its discussion in \secref{sec:dB}), the authors of \cite{Fiol-Yebra-Fabrega83}, \cite{Rauzy83} and \cite{Boshernitzan85} were apparently unaware of each other's work. One should note, however, that although formally all these papers deal with essentially the same setup of a subshift $\Si\subset\A^\Z$, the focus in \cite{Fiol-Yebra-Fabrega83}, \cite{Moreno03} (\exref{ex:kautz}, \exref{ex:tmc}, \exref{ex:sft}), on one hand, and in \cite{Rauzy83}, \cite{Boshernitzan85}, \cite{Cassaigne96} \linebreak (\exref{ex:rauzy}, \exref{ex:bosher}, \exref{ex:fpl}), on the other hand, is entirely different. It is well-known (e.g., see \cite{Lind-Marcus95}) that the subshifts of finite type considered in \cite{Fiol-Yebra-Fabrega83}, \cite{Moreno03} are in a sense completely opposite to the low complexity subshifts treated in \cite{Rauzy83}, \cite{Boshernitzan85}, \cite{Cassaigne96}. For instance, the former ones have a lot of periodic words, whereas the latter ones have none (cf. \secref{sec:periodic} below).
\end{rem}

\begin{ex} \label{ex:part}
Let $\au$ be a \emph{finite partition} of a state space $X$. Denote by $\A$ the set of elements of $\au$, and let $x\mapsto \au(x)\in\A$ be the map which assigns to any point $x\in X$ its \textsf{$\au$-name}, i.e.,
the element of the partition $\au$ which contains $x$. Then for any transformation $T:X\to X$ the \textsf{symbolic encoding map}
$$
\Sf: x \mapsto \bigl(\au(x),\au(Tx),\au(T^2x),\dots \bigr) \in \A^{\Z_+}
$$
is a semi-conjugacy between the original map $T$ and the shift transformation
$$
\Sb:\A^{\Z_+}\to \A^{\Z_+}
$$
(note that \emph{a priori}, without imposing any additional conditions, the image set $\Sf(X)\subset\A^{\Z_+}$ need not be closed or shift invariant). The map $\Sf$ is a standard tool in the theory of dynamical systems both in the measurable and in the topological setups (for instance, see \cite{Keane91} or \cite{Lind-Marcus95}).

As usual, let $T^{-n}\au$ denote the partition of $X$ into the $T^n$-preimages of the elements of the partition $\au$, so that $x,y\in X$ belong to the same element of $T^{-n}\au$ iff $T^n x$ and $T^n y$ belong to the same element of $\au$, i.e., $\au(T^n x)=\au(T^n y)$, and let
$$
\au_n= \au \vee T^{-1}\au \vee \dots \vee T^{-n+1}\au
$$
be the common refinement of the partitions $T^{-k}\au,\;0\le k\le n-1$. The $\au_n$-names of points $x\in X$ are length $n$ words in the alphabet $\A$, so that the triple $(X,T,\au)$ determines a subgraph $\ora\Ss^n_{\jj \au,T}$ of $\ora\B_{\ii\A}^n$ with the vertex set
$$
\{\au_n(x):x\in X\}\subset \A^n
$$
and the set of arrows
$$
\{\au_n(x)\mapstoto\au_n(Tx): x\in X\}
$$
which can be identified with
$$
\{\au_{n+1}(x):x\in X\}\subset \A^{n+1}
$$
(because $\au_n(x)$ and $\au_n(Tx)$ are the initial and the final segments of $\au_{n+1}(x)$, respectively).

Clearly, $\ora\Ss^n_{\jj \au,T}$ is completely determined just by the action of the shift transformation~$\Sb$ on the image set $\Sf(X)\subset\A^{\Z_+}$ with respect to the time 0 coordinate partition of $\A^{\Z_+}$, so that \exref{ex:fpl} is a particular case of this construction.
\end{ex}

\begin{ex} \label{ex:collatz}
Let
$$
\Cb(x) =
\begin{cases}
x/2 \;,& x\;\text{is even} \;, \\
(3x+1)/2\;,& x\;\text{is odd} \;,
\end{cases}
$$
be the \textsf{Collatz function} on the set of positive integers $\N$. The famous (and still very much open) \textsf{Collatz conjecture} claims that the $\Cb$-orbit of any starting point will eventually reach the number 1 (e.g., see \cite{Lagarias10}). Let $\au$ be the partition of $\N$ into even and odd numbers.
It is easy to see that if $x\equiv y\; (\!\!\!\!\mod 2^n)$, then $\au_n(x)=\au_n(y)$. Moreover, as it has been independently established by Terras \cite[Theorem 1.2]{Terras76} and Everett \cite[Theorem~1]{Everett77} (also see Lagarias \cite[Theorem B]{Lagarias85}), the map $x\mapsto\au_n(x)$ is a surjection, which implies that for any $n>0$ the associated graph $\ora\Ss^n_{\jj \au,\Cb}$ is the full de Bruijn graph $\ora\B_2^n$ \cite{Laarhoven-deWeger13}.
\end{ex}

\section{Periodic slider graphs: connectedness and step $d$ induced graphs} \label{sec:periodic}

All examples of factorial slider graphs described in \secref{sec:circularexamples} are based on using the linear word structure. However, our approach makes natural to consider circular words as well.

\begin{dfn}
For a subshift $\Si\subset\A^\Z$, let
$$
\Pc_n(\Si)\subset\Si
$$
be the set of its $n$-periodic words, and let
$$
\Si^{(n)}_\cc\subset\A^n
$$
be the set of all their length $n$ factors ($\equiv$ \textsf{$\Si$-admissible circular words of length} $n$). The associated induced circular slider graph $\ora\Ss\left[\Si^{(n)}_\cc\right]$ is called the \textsf{$n$-periodic slider graph} determined by the subshift $\Si$.
\end{dfn}

\begin{ex}
If $\Si$ is a topological Markov chain on $\A$, then $\Si^{(n)}_\cc$ consists of all words $\ab=\al_1 \al_2 \dots \al_n$ such that, in addition to all transitions $\al_i\al_{i+1},\; 1\le i\le n-1$, the transition $\al_n\al_1$ is also admissible.
\end{ex}

The periodic slider graphs corresponding to the Kautz topological Markov chain (no double letters, see \exref{ex:kautz} above) have been recently introduced by B\"ohmov\'a, Dalf\'o and Huemer \cite{Bohmova-Dalfo-Huemer15} and further studied by Dalf\'o \cite{Dalfo17a,Dalfo17}. We are not aware of any considerations of the periodic slider graphs for any other topological Markov chains or subshifts of finite type. All these subshifts have a lot of periodic words (as well as more general \emph{sofic subshifts} \cite{Lind-Marcus95} or \emph{subshifts of quasi-finite type} \cite{Buzzi05}), and the class of the associated periodic slider graphs should be quite interesting and promising for a future study.

The question about \emph{connectedness} does not really arise for the factorial slider graphs considered in \secref{sec:circularexamples}, as under the standard assumptions all these graphs are \textsf{strongly connected} in the sense that for any two vertices $x,y$ there exists a directed path from $x$ to $y$ (except for, possibly, \exref{ex:DNA} and \exref{ex:rauzy} of the languages $\Lc(\ww)$ generated by, respectively, a single finite or a single semi-infinite word~$\ww$). Indeed, in the case of \emph{sofic subshifts} (in particular, of topological Markov chains or subshifts of finite type) it follows from the usual \textsf{irreducibility} assumption (which consists in requiring that for any two words $\uu,\vv$ from the subshift language~$\Lc$ there is a word $\ww$ such that $\uu\ww\vv\in\Lc$, which is precisely what is needed for the strong connectivity of the associated slider graph). For the \emph{low complexity shifts} irreducibility follows from minimality, which is also a standard condition in this setup.

However, already for the periodic slider graphs (let alone the induced slider graphs determined by a general subset $\V\subset\A^n$) the situation is different, and these graphs need not be even \textsf{weakly connected} (i.e., as undirected graphs) in the simplest situations, for instance, for irreducible aperiodic topological Markov chains.

\begin{ex}
Let $\Si$ be the topological Markov chain on the 3-letter alphabet
$$
\A=\{\al,\be,\g\}
$$
with the admissible transitions
$$
\al \mapstoto \al \mapstoto \be \mapstoto \g \mapstoto \al \;.
$$
This chain has the property that its position at any moment of time $t$ is uniquely determined by the positions at times $t-1$ and $t+1$, i.e., for any two letters $\z_1,\z_2\in\A$ there exists at most one letter $\z\in\A$ which can be inserted between $\z_1$ and $\z_2$ in such a way that the word $\z_1\z\z_2$ is admissible. Therefore, in this situation the admissible de Bruijn transitions (\figref{fig:slidermoves}) on $\Si^{(n)}_\cc$ consist just in moving the slider along circular words without replacing any letters, so that the connected components of the undirected $n$-periodic slider graph $\ora\Ss\left[\Si^{(n)}_\cc\right]$ are just the rotation orbits (consisting of cyclic permutations) of $n$-periodic words. For instance, for $n=3$ there are just two periodic words $\al\al\al$ and $\al\be\g$ (up to a rotation $\equiv$ cyclic permutation), so that the 3-periodic slider graph has 4 vertices
$$
\al\al\al\;,\quad\al\be\g\;,\quad\be\g\al\;,\quad \g\al\be
$$
with the arrows
$$
\al\be\g \mapstoto \be\g\al \mapstoto \g\al\be \mapstoto \al\be\g \;, \quad \al\al\al \mapstoto \al\al\al \;,
$$
and it has two (weakly) connected components $\{\al\al\al\}$ and $\{\al\be\g,\be\g\al,\g\al\be\}$.
\end{ex}

However, one can easily modify \dfnref{dfn:slider} to make the periodic slider graphs connected by allowing slider transitions of uniformly bounded ``step length''. In spite of its naturalness, we could not find the following definition in the literature.

\begin{dfn}
Let $\G$ be a directed graph with the vertex set $V$, and let $X\subset V$. Fix a positive integer $d$.
Let us connect two vertices $x,y\in X$ with a directed arrow if there exists a path of length $\le d$ joining $x$ and $y$ in the ambient graph $\G$. The resulting digraph with the vertex set $X$ will be called the \textsf{step $d$} \textsf{graph on $X$ induced from $\G$} (or, \textsf{step $d$} \textsf{induced graph} in short). In the same way one defines \textsf{step} $d$ \textsf{undirected induced graphs} in the situation when the ambient graph is undirected. One can also talk about \textsf{strictly step} $d$ \textsf{induced graphs} when only the paths of length precisely $d$ in the ambient graph are considered. If the step $d$ induced graph on $X$ is connected (strongly or weakly), then we shall say that the set $X$ is \textsf{step $d$ connected} (resp., strongly or weakly) in the ambient graph~$\G$.
\end{dfn}

\begin{thm} \label{thm:connect}
If $\Si\subset\A^\Z$ is an irreducible sofic subshift on a finite alphabet~$\A$, then there exists an integer $d$ such that all periodic slider graphs $\ora\Ss\left[\Si^{(n)}_\cc\right]$ are step~$d$ connected in the corresponding ambient graphs $\ora\Ss\left[\Si^{(n)}\right]$.
\end{thm}

\begin{proof}
Since sofic subshifts are factors of finite type subshifts, it is sufficient to establish the claim just for the latter ones, so that by replacing, if necessary, the alphabet $\A$ with the alphabet $\A'=\A^N$ of the length~$N$ words for a certain integer $N$, the general case reduces in the usual way to the situation when $\Si$ is an irreducible topological Markov chain on $\A$ (cf. \cite{Lind-Marcus95}). For simplicity we shall also assume that this chain is aperiodic (otherwise one would have to add a couple of usual pretty obvious technicalities). Then the entries of a certain power of the admissibility matrix of $\Si$ are all positive, i.e., there exists a positive integer $\ka$ with the property that for any $\al,\be\in\A$ there is $\uu\in\A^\ka$ such that $\al\uu\be\in\Lc(\Si)$.

By the above, for any admissible circular word $\ab\in\Si^{(n)}_\cc$ and any letter $\al\in\A$, any  contiguous length $2\ka+1$ segment in $\ab$ can be replaced with a same length segment $\uu\al\vv,\, \uu,\vv\in\A^\ka,$ with the letter $\al\in\A$ in the middle, in such a way that the resulting new circular word $\ab'$ is also admissible. Then the set $\{\ab,\ab'\}$ is obviously step $(2\ka+2)$ connected in $\ora\Ss\left[\Si^{(n)}\right]$.

If $\bb\in\Si^{(n)}_\cc$ is another admissible circular word, then the above argument shows that $\ab$ is step $(2\ka+2)$ connected with an admissible circular word $\cb\in\Si^{(n)}_\cc$ which coincides with $\bb$ on a certain subset of positions $Z$ of the cycle $\Z_n$ such that the distance between any two neighbours $z,z'\in Z$ does not exceed $2\ka+1$. Since $\cb$ and $\bb$ are step $(2\ka+2)$ connected in $\ora\Ss\left[\Si^{(n)}\right]$, the claim follows with $d=2\ka+2$.
\end{proof}

\section{Missing links and transversally Markov circular slider graphs} \label{sec:transMark}

In spite of a number of studies of the \emph{fault tolerance} of de Bruijn and Kautz graphs to edge failures (e.g., see \cite{Ryu-Noel-Tang12}, \cite{Lin-Zhou-Li17} and the references therein), circular slider graphs other than the ones determined by various factorial languages (see \secref{sec:circularexamples}) have not attracted much attention \emph{per se}.

We shall give here the definition of a natural class of circular slider graphs, which is in an essential way based on using the circular word structure and can not be described by any factorial language.
As far as we know, this notion has not appeared in the literature so far.

\begin{dfn} \label{dfn:top}
Let $\Si$ be a topological Markov chain over a finite alphabet $\A$ determined by its set of admissible transitions $A=\Si^{(2)}\subset\A^2$. The associated \textsf{transversally Markov circular slider graph}
$$
\ora\Ss_{\ii\Si}(\A^n)
$$
of span $n$ is the circular slider graph with the vertex set $\A^n$, whose arrows are the de Bruijn transitions described on \figref{fig:slidermoves}, with the additional condition that $\al_1\al_{n+1}\in A$. In other words, one retains only the \textsf{transversally Markov de Bruijn transitions}, i.e., those for which the replacement $\al_1 \mapstoto \al_{n+1}$ is $\Si$-admissible. Moreover, given a subset $\V\subset\A^n$, one can further consider the \textsf{induced transversally Markov circular slider graph} $\ora\Ss_{\ii\Si}[\V]$, whose vertex set is $\V$, and whose arrows are the transversally Markov de Bruijn transitions between the words from $\V$.
\end{dfn}

More generally, one can also impose more complicated rules for admissibility of de Bruijn transitions, for instance, it can depend not just on the values of $\al_1$ and $\al_{n+1}$, but also on the letters in a certain fixed neighbourhood of the slider. Ultimately, when this neighbourhood becomes the whole circular word, one arrives at the definition of a general circular slider graph.

\begin{ex}
The \textsf{golden mean subshift} $\Si=\Si(A)$ is the topological Markov chain over the binary alphabet $\A=\{0,1\}$ with the forbidden transition $1\mapstoto 1$, i.e., with the set of admissible transitions $A=\{00,01,10\}$ (e.g., see \cite{Lind-Marcus95}). On the left of \figref{fig:golden} is the full slider graph $\ora\B_2^3$ of span 3 over the alphabet $\A=\{0,1\}$. Its vertices are all 3-letter words $\al_1\al_2\al_3$ in the alphabet $\A$, and its arrows are all de Bruijn transitions $\al_1\al_2\al_3\mapstoto\al_2\al_3\al_4$ labelled with the corresponding replacement letters $\al_4$. On the right of \figref{fig:golden} is the transversally Markov slider graph $\ora\Ss_{\ii\Si}\left(\A^3\right)$ obtained by removing the de Bruijn transitions $\al_1\al_2\al_3\mapstoto\al_2\al_3\al_4$ with $\al_1=\al_4=1$.

The induced subgraphs $\ora\Ss_{\ii\Si}\left[\Si^{(3)}\right]$ and $\ora\Ss_{\ii\Si}\left[\Si^{(3)}_\cc\right]$ of $\ora\Ss_{\ii\Si}\left(\A^3\right)$ determined by the subsets
$$
\Si^{(3)}= \{000, 001, 010, 100, 101\}
$$
and
$$
\Si^{(3)}_\cc = \{000, 001, 010, 100\}
$$
of all $\Si$-admissible and of $\Si$-admissible circular words, respectively, are presented on \figref{fig:golden2} \footnote{\;Of course, in general there is no need to use the same topological Markov chain both for defining a transversally Markov circular slider graph and for defining its induced subgraphs. However, there is not much choice in the case of a two letter alphabet.}.
\end{ex}

\begin{figure}[h]
\begin{center}
\raisebox{0.0cm}{\includegraphics[scale=.60]{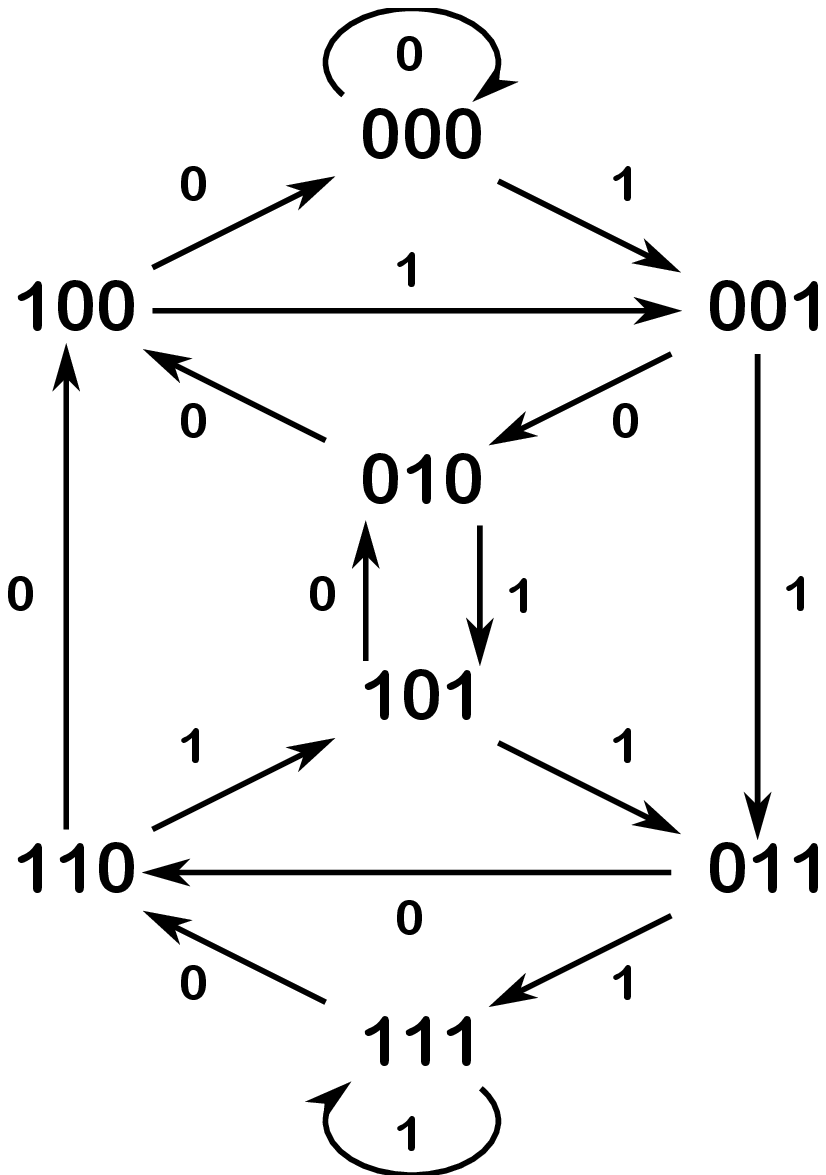}}
  \hspace*{2cm}
\raisebox{6.5mm}{\includegraphics[scale=.60]{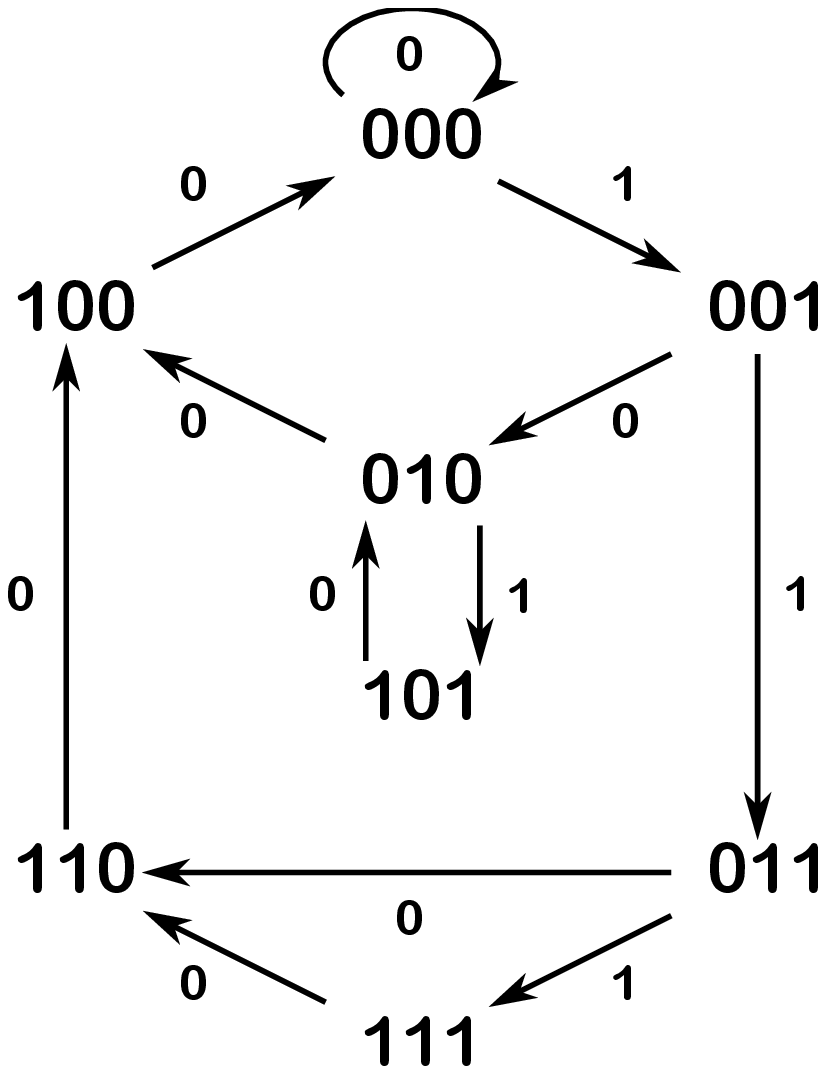}}
\end{center}
\caption{The full slider graph and the corresponding transversally Markov slider graph determined by the golden mean topological Markov chain.}
\label{fig:golden}
\end{figure}

\begin{figure}[h]
\begin{center}
\raisebox{0.0cm}{\includegraphics[scale=.60]{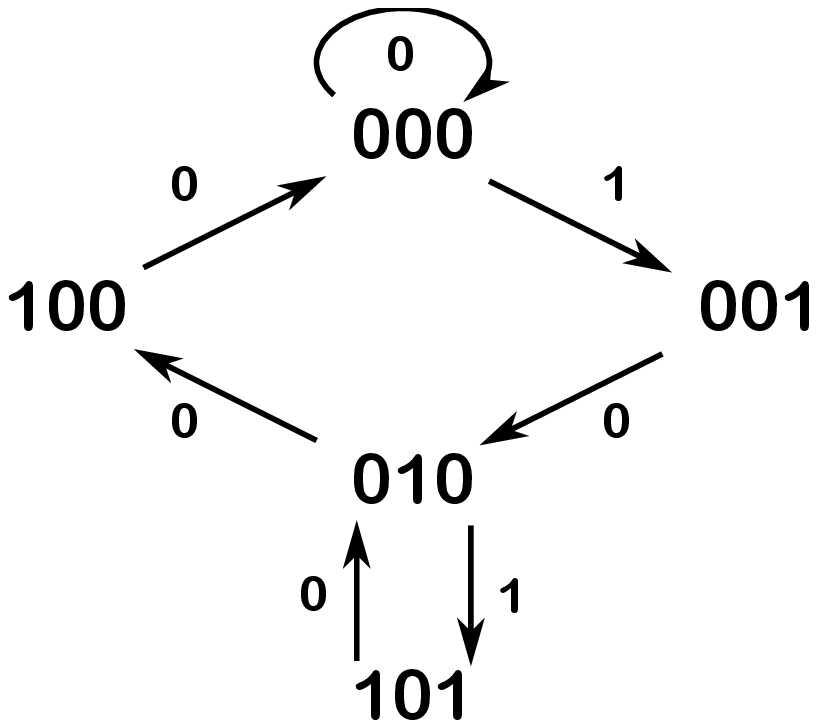}}
  \hspace*{2cm}
\raisebox{13.4mm}{\includegraphics[scale=.60]{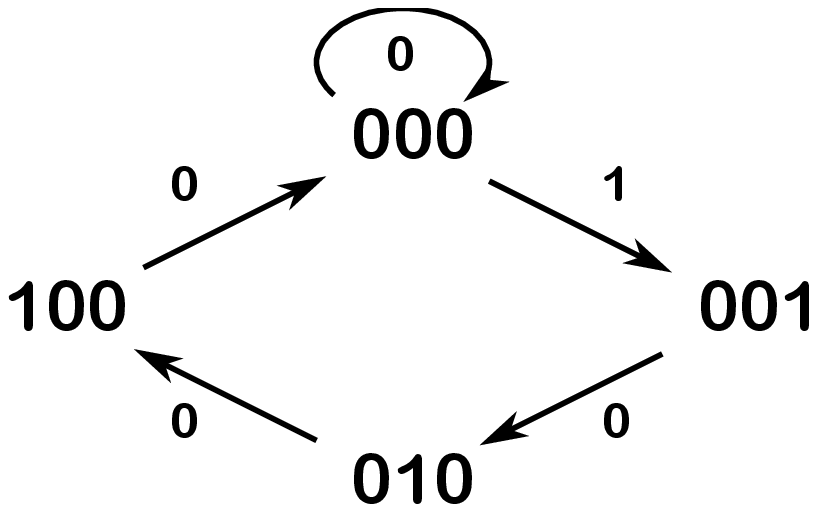}}
\end{center}
\caption{Induced subgraphs of the transversally Markov slider graph determined by the golden mean topological Markov chain.}
\label{fig:golden2}
\end{figure}

The following class of examples of transversally Markov circular slider graphs is inspired by the \emph{lamplighters} (see below \secref{sec:lamplighter}) and is based on the notion of the \emph{Cayley graph of a group} (the use of which is currently becoming popular in the theory of interconnection networks, see \cite{Ryu-Noel-Tang12}, \cite{Camelo-Papadimitriou-Fabrega-Vila14}). Let us first remind that the (directed) \textsf{Cayley graph} ($\equiv$ \textsf{Cayley topological Markov chain})
$$
\ora\C(G,K)
$$
on a group $G$ determined by a subset $K\subset G$ has the vertex set $G$ and the arrows
$$
g\mapstoto gk  \;, \qquad g\in G,k\in K \;.
$$
More generally, given a (right) action of a group $G$ on a set $X$, and a subset $K\subset G$, one defines the associated (directed) \textsf{Schreier graph} ($\equiv$ \textsf{Schreier topological Markov chain})
$$
\ora\Sch(X,K)
$$
with the vertex set $X$ and the arrows
$$
x\mapstoto xk  \;, \qquad x\in X,k\in K \;.
$$

\begin{dfn} \label{dfn:slidergroup}
Let $K$ be a subset of a group $G$. The associated \textsf{Cayley circular slider graph}
$$
\ora\Ss_{\ii K}(G^n) = \ora\Ss_{\ii \ora\C(G,K)}(G^n)
$$
is the transversally Markov circular slider of span $n$ over the alphabet $G$ determined by the Cayley topological Markov chain $\ora\C(G,K)$, i.e., its vertex set is $G^n$, and the arrows are
$$
(g_1,g_2,\dots,g_n) \mapstoto (g_2,g_3,\dots,g_n,g_1 h) \;,\qquad g_i\in G,\; h\in K \;.
$$
More generally, given a (right) action of a group $G$ on an action space $X$ and a subset $K\subset G$, the associated \textsf{Schreier circular slider graph}
$$
\ora\Ss_{\ii K}(X^n) = \ora\Ss_{\ii \ora\Sch(X,K)}(X^n)
$$
is the transversally Markov circular slider of span $n$ over the alphabet $X$ determined by the Schreier topological Markov chain $\ora\Sch(X,K)$, i.e., its vertex set is $X^n$, and the arrows are
$$
(x_1,x_2,\dots,x_n) \mapstoto (x_2,x_3,\dots,x_n,x_1 h) \;,\qquad x_i\in X,\; h\in K \;.
$$
\end{dfn}

Following \dfnref{dfn:top}, one can also consider the induced subgraphs $\ora\Ss_{\ii K}[\V]$ determined by various subsets $\V\subset G^n$ or $\V\subset X^n$ (cf. \secref{sec:circularexamples}).

\begin{rem}
In the case $K=G$ the graph $\ora\Ss_{\ii K}(G^n)$ coincides with the full slider graph $\ora\B_{\ii G}^n$, whereas for $K\subsetneqq G$ the graphs $\ora\Ss_{\ii K}(G^n)$ have the same vertex set~$G^n$ as $\ora\B_{\ii G}^n$, but fewer arrows. As we shall see below (\thmref{thm:schC}), the graphs $\ora\Ss_{\ii K}(G^n)$ are actually \emph{Schreier graphs of circular lamplighter groups}.
\end{rem}

\section{Lamplighters over cyclic groups} \label{sec:lamplighter}

Before discussing the relationship between slider graphs and lamplighter groups let us first remind the basic definitions concerning wreath products and lamplighters.

\begin{dfn}
The (restricted) \textsf{wreath product}
$$
G=A\wr B
$$
with the \textsf{active} (or, \textsf{base}) \textsf{group}~$A$ and the \textsf{passive group} (or, \textsf{group of states}) $B$ is the semi-direct product
$$
G = A\sd\fun(A,B)
$$
of the group $A$ and the group
$$
\fun(A,B) \cong \bigoplus_{a\in A} B
$$
of \textsf{finitely supported $B$-valued configurations} (i.e., those that take values different from the identity of $B$ at finitely many points only) on $A$ with the operation of pointwise multiplication (i.e., the direct sum of the copies of the group $B$ indexed by $A$), on which the group $A$ acts by translations.
\end{dfn}

In our notation for the wreath and semi-direct products the active group is always on the left (in accordance with the syntactic structure of the English language). However, quite often one also uses the notation in which $A$ and $B$ are switched. In this paper we are interested just in the situation when the \emph{active group $A$ is cyclic} (in particular, abelian, so that we shall use the additive notation for the group operation in~$A$).

In the context of functional and stochastic analysis the groups $\Z^d\wr\Z_2$ were first introduced by Vershik and the author \cite{Kaimanovich-Vershik83} under the name of the \textsf{groups of dynamical configurations}. Nonetheless, this term did not stick, and the current generally accepted standard is to call them \textsf{lamplighter groups} (apparently, this usage goes back to \cite{Lyons-Pemantle-Peres96a}). More general wreath products are also sometimes called lamplighter groups.

Below, if the cyclic group $A$ is finite (resp., infinite), we shall call $A\wr B$ a \textsf{circular} (resp., \textsf{linear}) \textsf{lamplighter group}.

\medskip

As a set, the group $G=A\wr B$ is the usual product of $A$ and the group of configurations $\fun(A,B)$. For a group element $(a,\F)\in G$ its $A$ component $a$ and its $\fun(A,B)$ component $\F$ are
usually referred to as the \textsf{lamlighter position} and the \textsf{lamp configuration}, respectively. The group operation in $G$ is ``skewed'' by using the left action of $A$ on $\fun(A,B)$ by the \textsf{group automorphisms}
$$
\T^a \F (x) = \F(x-a) \;,
$$
so that the \textsf{group multiplication} in $G$ is
%$$
\begin{equation} \label{eq:mult}
(a_1,\F_1) \cdot (a_2,\F_2) = ( a_1 + a_2, \F_1 \cdot \T^{a_1}\F_2 ) \;.
\end{equation}
%$$
The \textsf{identity} of $G$ is the pair $(0,\vn)$, where $\vn$ (the \textsf{identity of the group of configurations} $\fun(A,B)$) is the \textsf{empty configuration}
$$
\vn(a)=e \qquad \forall\,a\in A \;,
$$
and $e$ is the \textsf{identity of the group of states} $B$.

The \textsf{standard generators} of $G$ are
%$$
\begin{equation} \label{eq:standard}
(\pm 1,\vn) \quad \text{and} \quad \{(0, \de_0^b)\}_{b\in K} \;,
\end{equation}
%$$
where $K$ is a fixed (symmetric) generating set of $B$, and $\de_a^b\in\fun(A,B)$ denotes the configuration defined as
$$
\de_a^b (x) = \begin{cases}
b \;, & x=a \;,\\
e \;, & x \neq a
\end{cases}
$$
(in other words, $\de_a^b$ is the generator $b$ in the $a$-indexed copy of the group $B$ from the direct sum $\fun(A,B)\cong \bigoplus_{a\in A} B$).

Then, by \eqref{eq:mult}, for any $(a,\F)\in G$
$$
(a,\F) \cdot (\pm 1,\vn) = (a\pm 1,\F) \;,
$$
so that the right multiplication by the \textsf{walk generators} $(\pm 1,\vn)$ means that the ``lamplighter'' moves along $A$ one step to the left or to the right, whereas the lamp configuration $\F$ remains intact. In the same way,
$$
(a,\F) \cdot (0, \de_0^b) = (a,\F\cdot\de_a^b) \;,
$$
so that the right multiplication by the \textsf{switch generators} $(0,\de_0^b)$ means that the position $a$ of the lamplighter in the group $A$ remains the same, whereas the state of the lamp at $a$ changes from $\F(a)$ to $\F(a)\cdot b$.

\section{Lamplighters and circular slider graphs} \label{sec:lampslider}

We shall now introduce another generating set for the group $G=A\wr B$. Let
%$$
\begin{equation} \label{eq:generators}
\begin{aligned}
& \wt B_+ = \{ (1,\de_1^b) \}_{b\in B} = \{ (1,\vn) \cdot (0,\de_0^b) \}_{b\in B} \;, \\
& \wt B_- = \left(\wt B_+\right)^{-1} = \{ (-1,\de_0^b) \}_{b\in B} = \{ (0,\de_0^b) \cdot (-1,\vn) \}_{b\in B} \;, \\
& \wt B = \wt B_- \cup \wt B_+ \;.
\end{aligned}
\end{equation}
%$$

The Cayley graphs $\ora\C\left(G,\wt B_+\right)$ and $\ora\C\left(G,\wt B_-\right)$ of the group $G$ determined by the sets $\wt B_+$ and $\wt B_-$, respectively, are the digraphs obtained from the undirected Cayley graph $\C\left(G,\wt B\right)$ with respect to the symmetric generating set $\wt B$ by taking two opposite orientations of its edges.

By \eqref{eq:mult}, the result of the right multiplication of an element $(a,\F)$ by an increment $(1,\de_1^b)\in \wt B_+$ is
$$
(a,\F) \cdot (1,\de_1^b) = (a+1, \F\cdot\de_{a+1}^b ) \;,
$$
i.e., the lamplighter moves from the position $a$ to the position $a+1$, and the value of the configuration at the arrival point $a+1$ changes from $\F(a+1)$ to $\F(a+1)\cdot b$. In the same way, the right multiplication by $(-1,\de_0^b)\in \wt B_-$ amounts to changing the value of the configuration at the departure point $a$ from $\F(a)$ to $\F(a)\cdot b$ and moving from $a$ to $a-1$. Therefore, the generators from $\wt B_+$ (resp., from $\wt B_-$) can be called \textsf{walk-right---switch} (resp., \textsf{switch---walk-left}) \textsf{generators} (this nomenclature for various kinds of elements of lamplighter groups was, in the context of random walks, coined by Wolfgang Woess, and apparently first appeared in print in \cite{Lehner-Neuhauser-Woess08}).

Now, in order to make the picture more symmetric, it is convenient to ``mark'' the active group $A$ with the width 2 \textsf{slider} (sliding window) over the positions $a$ and $a+1$ (or, equivalently, to distinguish the edge between $a$ and $a+1$) rather than to point at lamplighter's position~$a$, see \figref{fig:window}.

\begin{figure}[h]
\begin{center}
$$
\begin{aligned}
& \hskip -2.8cm \cdots \hskip 1.2cm
\mbox{\includegraphics[scale=.02]{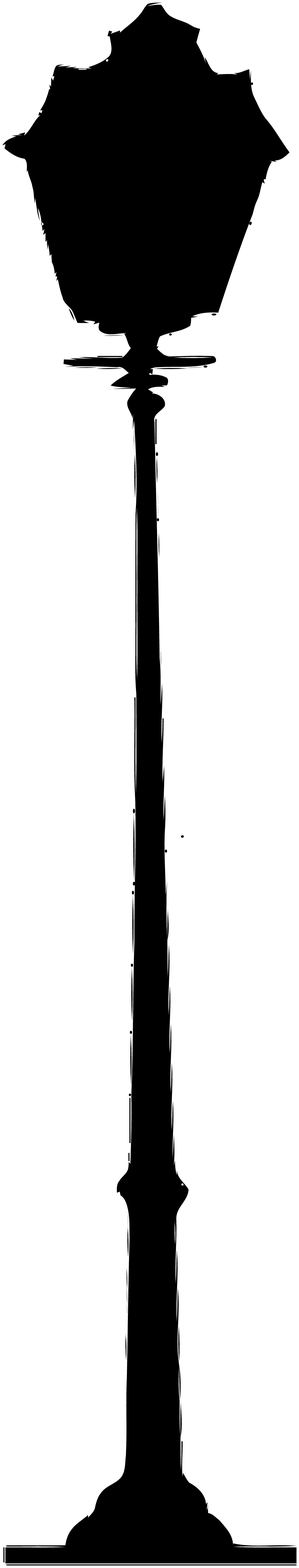}} \hskip 1.6cm \mbox{\includegraphics[scale=.02]{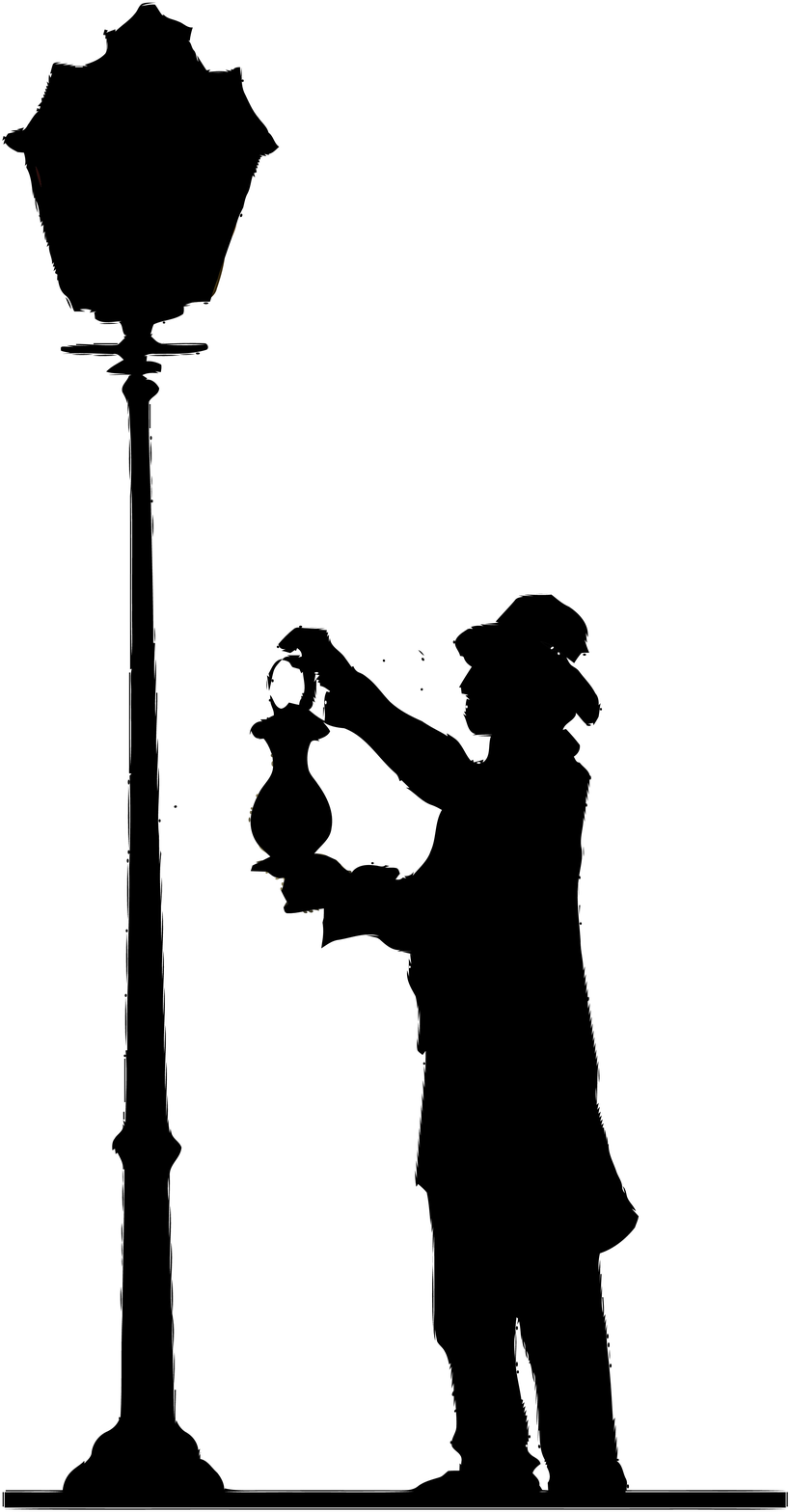}} \hskip 1.5cm
\mbox{\includegraphics[scale=.02]{lamppostlargenew10.eps}} \hskip 1.7cm
\mbox{\includegraphics[scale=.02]{lamppostlargenew10.eps}} \hskip 1.3cm
\cdots
\\
& {\hskip 1.6cm \DownArrow[20pt][>=latex,ultra thick]} \\
\cdots \quad \rule[1mm]{5mm}{0.4pt} \cf{${\scriptstyle \F(a-1)}$} & \rule[1mm]{5mm}{0.4pt} \myboxed{ \!\! \rule[1mm]{5mm}{0.4pt}\cf{${\scriptstyle \F(a)}$} \rule[1mm]{1cm}{1.5pt} \cf{${\scriptstyle \F(a+1)}$} \rule[1mm]{4mm}{0.4pt} } \hskip-1mm \rule[1mm]{5mm}{0.4pt} \cf{${\scriptstyle \F(a+2)}$} \rule[1mm]{5mm}{0.4pt} \quad \cdots
\end{aligned}
$$
\end{center}
          \caption{Lamplighter's window (slider).}
          \label{fig:window}
\end{figure}

Then the multiplication by an element $(1,\de_1^b)\in \wt B_+$ amounts to shifting the slider one position to the right and multiplying the state at the intersection of the old and the new sliders by $b$, see \figref{fig:right}.

\begin{figure}[h]
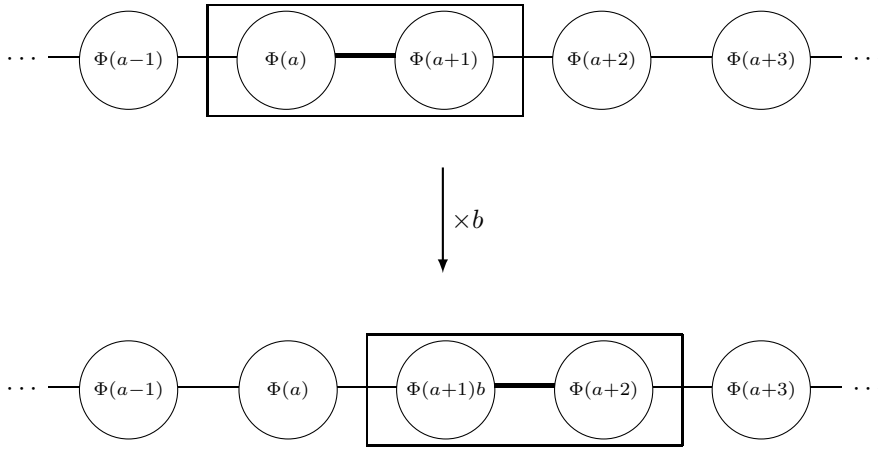

\begin{center}
$$
\begin{aligned}
\\
&\cdots \; \rule[1mm]{4mm}{0.4pt}
\cf{${\scriptstyle \F(a-1)}$}
\rule[1mm]{3.8mm}{0.4pt}
\myboxed{ \!\! \rule[1mm]{4.3mm}{0.4pt}
\cf{${\scriptstyle \F(a)}$}
\hskip -.1mm \rule[1mm]{8mm}{1.5pt} \hskip -.1mm
\cf{${\scriptstyle \F(a+1)}$}
\rule[1mm]{3.1mm}{0.4pt}
}
\hskip-1mm \rule[1mm]{4.8mm}{0.4pt}
\cf{${\scriptstyle \F(a+2)}$}
\rule[1mm]{8mm}{0.4pt}
\cf{${\scriptstyle \F(a+3)}$}
\rule[1mm]{4mm}{0.4pt} \; \cdots
\\ \\
& {\hskip 5.85cm \DownArrow[40pt][>=latex, thick]} \raisebox{6mm}{\;$\times b$}
\\ \\
&\cdots \; \rule[1mm]{4mm}{0.4pt}
\cf{${\scriptstyle \F(a-1)}$}
\rule[1mm]{8mm}{0.4pt}
\cf{${\scriptstyle \F(a)}$}
\rule[1mm]{3.8mm}{0.4pt}
\myboxed{ \!\! \rule[1mm]{4.3mm}{0.4pt}
\cf{${\scriptstyle \F(a+1)b}$}
\hskip -.1mm \rule[1mm]{8mm}{1.5pt} \hskip -.1mm
\cf{${\scriptstyle \F(a+2)}$}
\rule[1mm]{3.1mm}{0.4pt}
}
\hskip-1mm \rule[1mm]{4.8mm}{0.4pt}
\cf{${\scriptstyle \F(a+3)}$}
\rule[1mm]{4mm}{0.4pt} \; \cdots
\\
\phantom{z}
\end{aligned}
$$
\end{center}
          \caption{Walk-right---switch generators.}
          \label{fig:right}
\end{figure}

In a perfectly symmetrical way the multiplication by
$$
(-1,\de_0^b)=(1,\de_1^{b^{-1}})^{-1}\in \wt B_-
$$
amounts to shifting the slider one position to the left and multiplying the state at the intersection of the old and the new sliders by $b$, see \figref{fig:left}.

\begin{figure}[h]
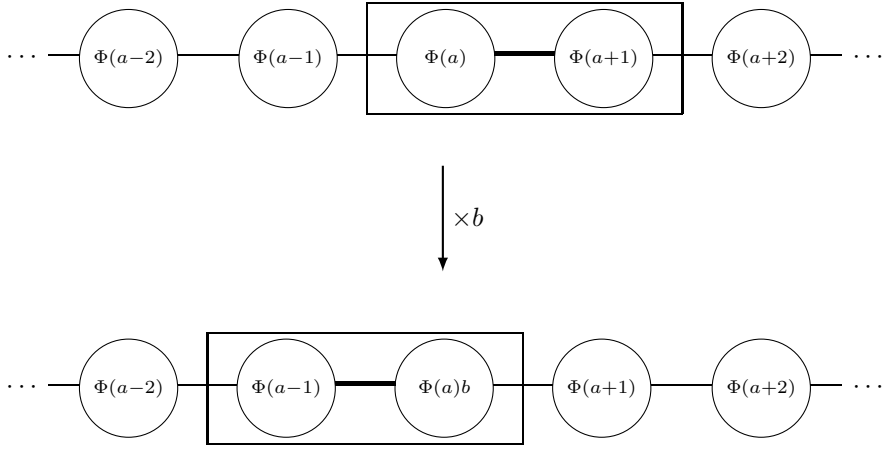

\begin{center}
$$
\begin{aligned}
\\
&\cdots \; \rule[1mm]{4mm}{0.4pt}
\cf{${\scriptstyle \F(a-2)}$}
\rule[1mm]{8mm}{0.4pt}
\cf{${\scriptstyle \F(a-1)}$}
\rule[1mm]{3.8mm}{0.4pt}
\myboxed{ \!\! \rule[1mm]{4.3mm}{0.4pt}
 \cf{${\scriptstyle \F(a)}$}
  \hskip -.1mm \rule[1mm]{8mm}{1.5pt} \hskip -.1mm
 \cf{${\scriptstyle \F(a+1)}$}
 \rule[1mm]{3.1mm}{0.4pt}
 }
\hskip-1mm \rule[1mm]{4.8mm}{0.4pt}
\cf{${\scriptstyle \F(a+2)}$}
\rule[1mm]{4mm}{0.4pt} \; \cdots \\
\\
& {\hskip 5.85cm \DownArrow[40pt][>=latex, thick]} \raisebox{6mm}{\;$\times b$}
\\
\\
&\cdots \; \rule[1mm]{4mm}{0.4pt}
\cf{${\scriptstyle \F(a-2)}$}
\rule[1mm]{3.8mm}{0.4pt}
\myboxed{ \!\! \rule[1mm]{4.3mm}{0.4pt}
\cf{${\scriptstyle \F(a-1)}$}
\hskip -.1mm \rule[1mm]{8mm}{1.5pt} \hskip -.1mm
\cf{${\scriptstyle \F(a)b}$}
\rule[1mm]{3.1mm}{0.4pt} }
\hskip-1mm \rule[1mm]{4.8mm}{0.4pt}
\cf{${\scriptstyle \F(a+1)}$}
\rule[1mm]{8mm}{0.4pt}
\cf{${\scriptstyle \F(a+2)}$}
\rule[1mm]{4mm}{0.4pt} \; \cdots
\\
\phantom{z}
\end{aligned}
$$
\end{center}
          \caption{Switch---walk-left generators.}
          \label{fig:left}
\end{figure}

Thus, the arrows of the Cayley graph $\ora\C\left(G,\wt B_+\right)$ on \figref{fig:right} look precisely like the de Bruijn transitions from \figref{fig:slidermoves}. An important difference, however, is that in the case of slider graphs the circular words in the alphabet $B$ are marked only once (with the position of the slider), whereas in the case of lamplighter groups the circular words are endowed with two pointers (both the position of the identity of the group $A$ and the position of the lamplighter). Therefore, in order to describe the slider graphs in terms of the lamplighter group one just has to eliminate the additional pointer, or, in other words, to eliminate the $A$ component of the elements $(a,\F)\in G=A\wr B$.

At the formal level, the group $A$ embeds into $G$ as the subgroup
$$
\{(a,\vn): a\in A\} \cong A \;.
$$
By formula \eqref{eq:mult},
$$
(a',\vn) (a,\F) = (a'+a, \T^{a'}\F) \qquad\forall\,a'\in A, (a,\F)\in G \;,
$$
which means that the left multiplication by $(a',\vn)$ amounts to shifting both lamplighter's position $a$ and the lamp configuration $\F$ by $a'$. Thus, the map
%$$
\begin{equation} \label{eq:change}
(a,\F) \mapsto \T^{-a}\F
\end{equation}
%$$
is constant on the cosets $Ag\subset G$, and therefore it allows one to identify the homogeneous space $A\bs G$ with the space of configurations $\fun(A,B)$, or, if the group~$A$ is finite, just with the set $B^n$ of the words of length $n=|A|$ in the alphabet $B$. Namely, the word $\bb\cong \T^{-a}\F$ is the sequence of the values of $\F$ read clockwise starting from the position adjacent to lamplighter's position $a$, i.e.,
$$
\bb = \be_1\be_2\dots\be_n \;, \quad\text{where}\quad\be_i=\F(a+i) \quad(\text{addition}\!\!\!\mod n) \;.
$$
In these terms the right action of the group $G$ on the space $A\bs G\cong B^n$ takes the form
%$$
\begin{equation} \label{eq:action}
\bb.(a,\F) = \bb. (0,\F) \cdot (a,\vn) = \T^{-a} (\bb\cdot\F) \;.
\end{equation}
%$$

The reason for the appearance of the minus sign in formula \eqref{eq:action} is that the map \eqref{eq:change} essentially consists in passing from the ``fixed coordinate system'' to the ``lamplighter coordinate system'' (it is the position of the lamplighter that becomes the reference point), and therefore, if the lamplighter moves in the clockwise direction (say, if $a=1$) in the fixed coordinate system, then the lamp configuration moves in the opposite anticlockwise direction with respect to the lamplighter (cf. \figref{fig:lettersmove} and \figref{fig:pointermoves}).

Then, as it follows from comparing \figref{fig:right} and \figref{fig:slidermoves}, the right action of the elements from $\wt B_+$ on $A\bs G\cong B^{|A|}$ consists precisely in de Bruijn transitions. Thus, we have proved

\begin{thm} \label{thm:sch}
Let $A$ be a finite cyclic group. Then the Schreier digraph
$$
\Sch\left(A\bs G,\wt B_+\right) \cong \Sch\left(B^{|A|},\wt B_+\right)
$$
of the action \eqref{eq:action} of the wreath product $G=A\wr B$ on the homogeneous space
$A\bs G\cong B^{|A|}$ with respect to the generating set $\wt B_+$ \eqref{eq:generators} is isomorphic to the de Bruijn digraph $\ora\B_{\ii B}^{|A|}$ of span $|A|$ over the alphabet $B$.
\end{thm}

\begin{rem} \label{rem:gln}
A realization of de Bruijn graphs as Schreier graphs of the \emph{infinite} lamplighter groups $\Z\wr\Z_d$ was recently obtained by Grigorchuk, Leemann and Nagnibeda \cite[Theorem 4.4.1]{Grigorchuk-Leemann-Nagnibeda16}. In a sense, their Theorem 6.1.3 implicitly contains our result as it identifies the Cayley graph of the finite lamplighter group $\Z_n\wr\Z_d$ with the \emph{spider-web graph} built from the de Bruijn graph $\ora\B_{\ii d}^n$ by taking its direct (or tensor) product with $\Z_n$ (see below \secref{sec:spiderslider}). However, our argument is much more direct.
\end{rem}

If one considers the generating set \eqref{eq:generators}, then the arising Cayley and Schreier graphs only depend on the size of the group $B$. Of course, this is no longer the case if the set $\wt B_+$ \eqref{eq:generators} is replaced with
%$$
\begin{equation} \label{eq:genK}
\wt K_+=\{ (1,\de_1^b) \}_{b\in K} \;,
\end{equation}
%$$
for a proper subset $K\subsetneqq B$. However, the same argument as above still yields

\begin{thm} \label{thm:schC}
Let $A$ be a finite cyclic group, and $K\subset B$. Then the Schreier digraph
$$
\Sch\left(A\bs G,\wt K_+\right) \cong \Sch\left(B^{|A|},\wt K_+\right)
$$
of the action \eqref{eq:action} of the wreath product $G=A\wr B$ on the homogeneous space
$A\bs G\cong B^{|A|}$ with respect to the generating set $\wt K_+$ \eqref{eq:genK} is isomorphic to the Cayley circular slider digraph $\ora\Ss_{\ii K}\left(B^{|A|}\right)$ from \dfnref{dfn:slidergroup}.
\end{thm}

One can also obtain a similar description of the Schreier circular slider graphs from \dfnref{dfn:slidergroup} in terms of Schreier graphs of lamplighter groups.

\section{Spider slider graphs} \label{sec:spiderslider}

We shall denote by $\ora\CC$ the Cayley digraph of a cyclic group $C$ determined by the generator 1, so that the arrows of $\ora\CC$ are $i\mapstoto i+1$ (addition\!\!\! $\mod |C|$). Let us recall that the \textsf{tensor} (or, \textsf{direct}) \textsf{product of digraphs} $\G_1,\G_2$ is the digraph $\G_1\otimes\G_2$ such that both its vertex set and its edge set are the products, respectively, of the vertex sets and of the edge sets of $\G_1$ and $\G_2$, with the natural incidence relations (e.g., see \cite{Hammack-Imrich-Klavzar11}). As we have already mentioned in \remref{rem:gln}, the tensor products $\ora\CC \otimes\ora\B_{\ii m}^n$ are known as \textsf{spider-web graphs} (see Grigorchuk -- Leemann -- Nagnibeda \cite{Grigorchuk-Leemann-Nagnibeda16} and the references therein). They are parameterized by the alphabet size $m$, the span $n$ of the de Bruijn graph, and by the size $|C|$ of the group~$C$.

Theorem 6.1.3 of Grigorchuk -- Leemann -- Nagnibeda \cite{Grigorchuk-Leemann-Nagnibeda16} identifies, by using an algebraic approach based on a classification of all subgroups of the infinite lamplighter groups
$\Z\wr\Z_m$ earlier obtained by Grigorchuk -- Kravchenko \cite{Grigorchuk-Kravchenko14}, the spider-web graphs $\ora\CC \otimes\ora\B_{\ii m}^n$ with the Cayley graphs of certain finite groups for a number of combinations of the parameters $m,n,|C|$. In particular, if $|C|=n$, they establish an isomorphism of $\ora\CC_n \otimes\ora\B_{\ii m}^n$ with the Cayley graph of the finite lamplighter group $\Z_n\wr \Z_m$ endowed with the generating set \eqref{eq:generators}.

As we have just explained in \thmref{thm:sch} and \thmref{thm:schC}, the Cayley circular slider graphs $\ora\Ss_{\ii K}\left(B^{|A|}\right)$ (in particular, the de Bruijn graphs $\ora\B_{\ii B}^{|A|}$) can be obtained from the Cayley graphs of circular lamplighter groups by removing an extra pointer (which amounts to passing to the corresponding Schreier graph). Conversely, in order to recover the Cayley graph from the corresponding de Bruijn graph one just has to add a parameter $a\in A$ describing the relative position of the slider and of the identity of the group $A$, which yields

\begin{thm}\label{thm:spider}
Let $A$ be a finite cyclic group, and $K\subset B$. Then the Cayley graph $\C\left(A\wr B,\wt K_+\right)$ is isomorphic to the tensor product $\ora A\otimes\ora\Ss_{\ii K}\left(B^{|A|}\right)$.
\end{thm}

In particular, for $K=B=\Z_m$ we recover (in a much more direct way) the aforementioned result from \cite[Theorem 6.1.3]{Grigorchuk-Leemann-Nagnibeda16}.

The construction of spider-web graphs obviously carries over to arbitrary circular slider graphs:

\begin{dfn}
We shall call the tensor product $\ora A\otimes\ora\Ss$ of the Cayley digraph $\ora A$ of a cyclic group $A$ and of a circular slider graph $\ora\Ss$ a \textsf{spider slider graph}.
\end{dfn}

Numerous interesting properties of the classical spider-web graphs discussed in
\cite{Grigorchuk-Leemann-Nagnibeda16} and, in more detail, in \cite{Leemann16} suggest that it would be interesting to study more general spider slider graphs associated with the slider graphs considered in \secref{sec:circularexamples}, \secref{sec:periodic} and \secref{sec:transMark}.

\bibliographystyle{amsalpha}
%\bibliography{C:/S/MyTEX/mine}
\bibliography{kaimanovich.bbl}

\end{document}